\documentclass[reqno, a4paper, 10pt]{amsart}
\usepackage{amssymb}
\usepackage{amsmath}
\usepackage[mathscr]{euscript}
\usepackage{hyperref}
\usepackage{graphicx}
\usepackage[normalem]{ulem}

\usepackage{amssymb}
\usepackage{hyperref}
\RequirePackage[dvipsnames]{xcolor} 
\definecolor{halfgray}{gray}{0.55} 
\definecolor{webgreen}{rgb}{0,0.5,0}
\definecolor{webbrown}{rgb}{.6,0,0} \hypersetup{%
  colorlinks=true, linktocpage=true, pdfstartpage=3,
  pdfstartview=FitV,%
  breaklinks=true, pdfpagemode=UseNone, pageanchor=true,
  pdfpagemode=UseOutlines,%
  plainpages=false, bookmarksnumbered, bookmarksopen=true,
  bookmarksopenlevel=1,%
  hypertexnames=true,
  pdfhighlight=/O,
  urlcolor=webbrown, linkcolor=RoyalBlue,
  citecolor=webgreen, 
  pdftitle={},%
  pdfauthor={},%
  pdfsubject={2000 MAthematical Subject Classification: Primary:},%
  pdfkeywords={},%
  pdfcreator={pdfLaTeX},%
  pdfproducer={LaTeX with hyperref}%
}

\usepackage{enumitem}

\usepackage{orcidlink}

\theoremstyle{plain}
\numberwithin{equation}{section}
\newtheorem{theorem}{Theorem}
\newtheorem{proposition}{Proposition}
\newtheorem{lemma}{Lemma}
\newtheorem{corollary}{Corollary}
\newtheorem{definition}{Definition}
\theoremstyle{remark}
\newtheorem{remark}{Remark}
\newtheorem{example}{Example}
\renewcommand{\epsilon}{\varepsilon}
\renewcommand{\phi}{\varphi}

\DeclareMathOperator{\Ima}{Im}
\DeclareMathOperator{\Ker}{Ker}

\DeclareMathOperator{\diag}{diag}
\def\N{\mathbb{N}}
\def\cA{\EuScript{A}}
\def\cB{\EuScript{B}}
\def\R{\mathbb{R}}

\def\Id{\text{\rm Id}}
\def\Z{\mathbb{Z}}

\begin{document}

\title [Smooth linearization of nonautonomous dynamics]{Smooth linearization of nonautonomous dynamics under general dichotomic behaviour}

\author[Lucas Backes]{Lucas Backes \orcidlink{0000-0003-3275-1311}}
\address[Lucas Backes]
{Department of Mathematics\\
    Universidade Federal do Rio Grande do Sul\\
    Av. Bento Goncalves 9500, CEP 91509-900, Porto Alegre, RS, Brazil}
\email[L. Backes]{lucas.backes@ufrgs.br}

\author[Davor Dragi\v cevi\'c]{Davor Dragi\v cevi\'c  \orcidlink{0000-0002-1979-4344}}
\address[Davor Dragi\v cevi\'c]
{Faculty of Mathematics\\
    University of Rijeka\\
    Rijeka 51000, Croatia}
\email[D.~Dragi\v cevi\'c]{ddragicevic@math.uniri.hr}

\author[Wenmeng Zhang]{Wenmeng Zhang \orcidlink{0000-0002-1176-4140}}
\address[Wenmeng Zhang]
{School of Mathematical Sciences\\
    Chongqing Normal University\\
    Chongqing 401331, P.R.China}
\email[W.~M.~Zhang]{wmzhang@cqnu.edu.cn}

\date{\today}


\begin{abstract}
The main purpose of this paper is to formulate new conditions for smooth linearization of nonautonomous systems with discrete and continuous time. Our results assume that the linear part admits a very general form of dichotomy known as $\mu$-dichotomy and that the associated $\mu$-dichotomy spectrum exhibits appropriate spectral gap and spectral band conditions. We observe that our notion of $\mu$-dichotomy encompasses the classical notions of exponential, polynomial and logarithmic dichotomies as very particular cases. In particular, our result is in sharp contrast to  most of the previous results in the literature which  assumed that the linear part admits an exponential dichotomy. Our techniques exploit the relationship between $\mu$-dichotomy and exponential dichotomy via a suitable reparametrization of time.
\end{abstract}

\maketitle

\section{Introduction} 

One of the most important notions in the theory of dynamical systems is that of (uniform) hyperbolicity. It plays a fundamental role in many situations, like in the study of stability, chaos, and bifurcations. The counterpart of this notion  in the case of nonautonomous dynamical systems is the \emph{exponential dichotomy}, a concept whose study dates back to the pioneering work of Perron \cite{Per-30}. In the discrete time case, we say that the nonautonomous dynamics defined by the sequence $(A_n)_{n\in \Z}$ of invertible linear operators on $\R^d$ admits an exponential dichotomy
if there exist constants ${K},\lambda>0$ and a sequence of projections $(P_n)_{n\in \Z}$ on $\R^d$ such that for every $n,m\in \Z$ with $m>n$, $A_nP_n=P_{n+1}A_n$ and 
\begin{align}\label{ED-00}
\| A_{m-1}\cdots A_n P_n\| \le  {K}e^{-\lambda(m-n)}\;\text{ and }\;
\| A^{-1}_n\cdots A^{-1}_{m-1} Q_m\| \le  {K} e^{-\lambda (m-n)}
\end{align}
where $Q_m=\Id-P_m$.
As with hyperbolicity for autonomous systems, the notion of an exponential dichotomy has broad applications, including the study of invariant manifolds, invariant foliations, and normal forms \cite{BV-book,Pot-book,SY-book}, and there exists extensive literature regarding this notion for both finite and infinite-dimensional systems (see, for example, \cite{CL-JDE95,HZ-JDE04,P-JDE84,ZZ-JFA16}).

Even though the notion of an exponential dichotomy is of great importance in the study of nonautonomous systems, there are many important classes of systems that do not fit this framework. 
In fact, due to the flexibility of the nonautonomous setting, it is easy to construct broad classes of systems which admit a splitting into stable and unstable directions, but with non necessarily exponential rates of contraction and expansion (see, for instance, Example \ref{example: main theo discrete} below). In this paper, we are interested in systems with this type of general behavior. 
Given any strictly increasing sequence $ \mu =(\mu_n)_{n\in \Z}$ of positive numbers with $\mu_0=1$ and such that 
\[
\lim_{n\to +\infty}\mu_n=+\infty \; \text{ and } \;  \lim_{n\to -\infty}\mu_n=0,
\]
we say that $(A_n)_{n\in \Z}$ admits a \emph{$\mu$-dichotomy} if there exist constants and a sequence of projections as in the case of an exponential dichotomy such that, instead of satisfying \eqref{ED-00}, they satisfy 
\begin{align*}
\| A_{m-1}\cdots A_n P_n\| \le  K \left(\frac{\mu_m}{\mu_n}\right)^{-\lambda} \;\text{ and }\;
\| A^{-1}_n\cdots A^{-1}_{m-1} Q_m\| \le   K \left(\frac{\mu_m}{\mu_n}\right)^{-\lambda}
\end{align*}
for $m\geq n$. Note that by taking $\mu_n=e^n$, $n\in \Z$, we recover the notion of an exponential dichotomy. Other particular cases that are widely explored in the literature correspond to the case where the sequence $\mu$ has a polynomial or logarithmic behavior. We will expand on this in the following sections.

In what follows, we are interested in studying the problem of smooth linearization of  nonautonomous systems whose linear part admits a $\mu$-dichotomy. We will consider systems with both discrete and continuous time.  This problem consists of finding a time dependent smooth change of coordinates sending the solutions of the nonlinear system into the solutions of the linear one. As a preliminary to our main results, we formulate smooth linearization results for one-sided systems whose linear part admits an exponential dichotomy; these are derived from existing results for two-sided systems \cite{DZZ,DZZ20}. Then, by combining these findings with a suitable time reparametrization (as explored in \cite{DS}), we obtain the desired results in the $\mu$-dichotomic context. A key aspect for the validity of our results is that the associated \emph{$\mu$-dichotomy spectrum} exhibits appropriate spectral gap and spectral band conditions.
The importance of our results stems from their broad generality. In fact, we are able to treat in a unified manner several important classes of systems for which results of this type were previously unavailable. We would like to point out the relevant works~\cite{BV, LX} dealing with nonautonomous linearization under $\mu$-dichotomies. However, these works do not deal with  smooth or differentiable linearization.

The linearization problem described above is among the most fundamental ones in the theory of dynamical systems and has a long history (of which we do not attempt to provide a complete survey). In the context of autonomous dynamics, the problem of establishing sufficient conditions under which conjugacy exhibits higher regularity was first considered in the pioneering works of Sternberg~\cite{Sternberg2, Sternberg58}. In the 1950s, he proved that for each $r\in \mathbb{N}$, there is an integer $k\in \mathbb{N}$ such that $C^k$ hyperbolic diffeomorphisms satisfying some non-resonance conditions up to order $k$ admit a $C^r$ linearization.
We also mention the fundamental works of Hartman \cite{HartPAMS60} and Grobman \cite{Grobman}, who proved independently in the 1960s that $C^1$ hyperbolic diffeomorphisms can be $C^0$ linearized near the hyperbolic ﬁxed point. 
This result shows that if one is only interested in the case $r=0$ in Sternberg's theorem, one can take $k=1$ and, moreover, no non-resonance conditions are required (see Section \ref{sec: lin exp dich} for details). We note that infinite-dimensional versions of this result were obtained independently by Palis \cite{Palis} and Pugh \cite{Pugh}.

Palmer \cite{Palmer} established the first nonautonomous version of the Grobman-Hartman theorem for continuous-time dynamics, while a discrete-time version was subsequently formulated in \cite{AW}. Both results were obtained by assuming that the associated linear system admits a uniform exponential dichotomy.
The problem of \emph{smooth} linearization for nonautonomous systems has been considered only recently. To our knowledge, the first results in this direction were obtained in \cite{CMR,CR} for uniformly and nonuniformly exponentially stable systems. Moreover, in \cite{CDS}, the authors established a Sternberg-type theorem in the setting of a uniform exponential dichotomy with continuous time.
More relevant to our context are the works \cite{DZZ,DZZ20} (see also~\cite{D1}) where $C^1$, differentiable (at 0) and H\"older linearization results were obtained in the setting of a nonuniform exponential dichotomy under some spectral gap and spectral band conditions. As previously emphasized, these results will play a central role in our proofs.

Finally, we mention  recent results \cite{BD, CHR, CJ, J} dealing with the smooth linearization of nonautonomous systems in the absence of  any kind of dichotomy, non-resonance or spectral gap assumptions. However, as observed in~\cite{BD},  such results in the setting of~\cite{DZZ} (and consequently also in the setting of the present paper)  can fail to be applicable or could yield weaker results.

This paper is organized as follows. The first three sections are dedicated to discrete-time dynamical systems. More precisely, in Section \ref{sec: pol dich}, we introduce the notions of strong $\mu$-dichotomy and the strong $\mu$-dichotomy spectrum, exploring the relationship between these concepts and exponential dichotomy via time reparametrization. In Section \ref{sec: lin exp dich}, we present results concerning $C^1$, differentiable (at 0), and Hölder linearization under the assumption of an exponential dichotomy on the half-line. Section \ref{sec: lin pol} is devoted to establishing the $C^1$, differentiable (at 0), and Hölder linearization results under the assumption of a $\mu$-dichotomy. At the end of Section \ref{sec: lin pol}, we briefly explain how these results can be extended to the infinite-dimensional setting. Finally, in Section \ref{sec: continuous time case}, we provide continuous-time versions of the results established in Section \ref{sec: lin pol}. 
Appendix \ref{sec: appendix a} serves as an addendum to the proof of the main results in \cite{DZZ}. In this section, we strengthen the properties of the conjugacies presented in the aforementioned work and provide an argument that was missing in the original paper.

\vspace{0.1in}
\textbf{Note:} We observe that a previous version of this paper was written in the particular case when the linear part of the system admits a (nonuniform) polynomial dichotomy. That version is available on ArXiv (\href{https://arxiv.org/abs/2210.04804}{https://arxiv.org/abs/2210.04804}) and will remain as a permanent preprint.

\section{$\mu$-dichotomy and dichotomy spectrum} \label{sec: pol dich} 
In this section, we introduce the notions of $\mu$-dichotomy and $\mu$-dichotomy spectrum, followed by a discussion of their fundamental properties. We start by fixing some notation.

Let $\N$ denote the set of all natural numbers, $\Z$ the set of integers, $\Z^+$ the set of nonnegative integers, and $\Z^-$ the set of nonpositive integers. Given $J\in \{\Z, \Z^-, \Z^+\}$ and a sequence $(A_n)_{n\in J}$ of invertible operators on $\R^d$, let us consider the associated nonautonomous linear system
\begin{equation}\label{eq: LE}
    x_{n+1}=A_nx_n, \quad n,n+1\in J
\end{equation}
whose evolution family is given by
\begin{equation}\label{eq: cocycle}
\cA(m,n):=\begin{cases}
A_{m-1} \cdots A_n, & m>n; \\
\Id, & m=n;\\
A_m^{-1} \cdots A_{n-1}^{-1}, & m<n,
\end{cases}
\end{equation}
where $\Id$ denotes the identity operator on $\R^d$.

\subsection{Growth rates and $\mu$-dichotomy for discrete time dynamics}
Let $ \mu =(\mu_n)_{n\in J}$ be a strictly increasing sequence of positive numbers with $\mu_0=1$ and such that 
\begin{itemize}
\item for $J=\Z^{+}$,
\begin{equation}\label{eq: growth +}
     \lim_{n\to +\infty}\mu_n=+\infty;
     \end{equation}
\item for $J=\Z ^-$,
\begin{equation}\label{eq: growth -}
     \lim_{n\to -\infty}\mu_n=0;
     \end{equation}
\item for $J=\Z$, conditions \eqref{eq: growth +} and \eqref{eq: growth -} are satisfied.
\end{itemize}
Any sequence $\mu$ satisfying these properties is called a \emph{growth rate}.

\begin{definition}\label{def: mu dichot}
Let $(A_n)_{n\in J}$ be a sequence of invertible operators on $\R^d$. We say that $(A_n)_{n\in J}$ admits a \emph{$\mu$-dichotomy} if there exists a family of projections $(P_n)_{n\in J}$ on $\R^d$ and constants $K,\lambda >0$ such that
\begin{itemize}
\item  for $n,n+1\in J$,
\begin{equation}\label{pro}
A_nP_n=P_{n+1}A_n;
\end{equation}
\item for $m,n\in J$ with $m\geq n$,
\begin{equation}\label{pd1}
\|\cA(m,n)P_n\|\leq K \left(\frac{\mu_m}{\mu_n}\right)^{-\lambda} \quad \text{ and }\quad 
\|\cA(n,m)Q_m\|\leq K \left(\frac{\mu_m}{\mu_n}\right)^{-\lambda} \quad 
\end{equation}
where $Q_n:=\Id -P_n$.
\end{itemize}
In addition, if there exists $a\geq \lambda$ such that
\begin{equation}\label{bg}
\|\cA(m,n)\|\leq K \left(\frac{\mu_m}{\mu_n}\right)^{a} \quad \text{ and } \quad 
\|\cA(n,m)\|\leq K \left(\frac{\mu_m}{\mu_n}\right)^{a} 
\end{equation}
for $m\geq n$, we say that $(A_n)_{n\in J}$ admits a \emph{strong $\mu$-dichotomy}.
\end{definition}

\begin{remark}\label{remark: particular cases of dich}
We observe that the notion of $\mu$-dichotomy generalizes several well-known notions of dichotomy. For example, in the case where $J=\N$, by taking $\mu_n=e^{n}$, $n\in \N$, we recover the notion of \emph{exponential dichotomy} (see \cite{Cop-book}); by taking $\mu_n=1+n$, $n\in \N$, we recover the notion of \emph{polynomial dichotomy} (see \cite{BS1,DSS, D}); by taking $\mu_n= \ln(e+n)$, $n\in \N$, we recover the notion of \emph{logarithmic dichotomy} (see \cite{Silva}).
\end{remark}

\begin{remark}
 We emphasize that several versions of the $\mu$-dichotomy for both discrete- and continuous-time dynamics have already appeared in the literature, and various properties of these systems have already been investigated. For instance, there are studies addressing invariant manifolds \cite{BS-2, BS-1, Pan}, the shadowing property \cite{BD1}, spectral properties \cite{Bac25, JG, Silva2}, admissibility \cite{BD2, Silva}, reducibility \cite{CJ, Silva}, expansivity~\cite{D, EPR} and roughness \cite{NP, CZQ, Chu} for systems exhibiting this type of behavior.
\end{remark}

\subsection{Time rescaling} \label{sec: time rescaling}

We now recall a time rescaling procedure described in \cite{DS}. Given a growth rate $\mu=(\mu_n)_{n \in \Z^+}$, we can associate to it the continuous, strictly increasing function $\widetilde \mu\colon [0,+\infty) \to [1,+\infty )$ given by
\[
\widetilde \mu(t)=
\begin{cases}
	\mu_n & \text{ if } \ t=n\\
	\mu_n+(t-n)(\mu_{n+1}-\mu_n) & \text{ if } \ n<t<n+1
\end{cases}, \quad n \in \Z^+.
\]
One can easily check that this is an invertible function and that its inverse ${\widetilde \mu}^{-1}\colon [1,+\infty) \to [0,+\infty)$ is the continuous, strictly increasing function given by
\[
\widetilde \mu^{-1}(t)=
\begin{cases}
	n & \text{ if } \ t=\mu_n\\
	n+\dfrac{t-\mu_n}{\mu_{n+1}-\mu_n} & \text{ if } \ \mu_n<t<\mu_{n+1}
\end{cases}, \quad n \in \Z^+.
\]
Using these maps, we associate with the system~\eqref{eq: LE} the family of invertible linear operators $(B^{\mu}_n)_{n \in \mathbb{N}}$ given by 
\begin{equation}\label{eq:B-mu}
B^{\mu}_n=\cA\left(\lfloor \widetilde\mu^{-1}(e^n) \rfloor +1, \  \lfloor \widetilde\mu^{-1}(e^{n-1}) \rfloor +1 \right), \quad \forall n \in \mathbb{N}
\end{equation}
and consider the nonautonomous linear system
\begin{equation}\label{eq:Lin-B}
	y_{n+1}=B^{\mu}_n \, y_n, \quad n \in \mathbb{N}. 
\end{equation}
It is easy to see that the evolution family associated to \eqref{eq:Lin-B} is given by 
\[
\begin{split}
\cB ^{\mu}(m,n)
 =\cA(\lfloor\widetilde\mu^{-1}(e^{m-1})\rfloor +1,\,\lfloor\widetilde\mu^{-1}(e^{n-1})\rfloor +1).
\end{split}
\]

We have the following relationship between systems \eqref{eq: LE} and \eqref{eq:Lin-B}.
\begin{theorem}[Corollary 3.3 of \cite{DS}] \label{theo: equiv exp mu dich}
Let $\mu$ be a growth rate satisfying 
\begin{equation}\label{eq: grbound}
	\frac{\mu_{n+1}}{\mu_n}\le \theta, \quad \text{for all} \quad n \in \N,
\end{equation} 
for some $\theta\ge 1$ and consider a sequence of invertible linear operators $(A_n)_{n\in \N}$ on $\R^d$ for which there exist $K, a>0$ such that \eqref{bg} is satisfied. Then, there exists $K'>0$ such that
\begin{equation*}
\|\cB ^\mu(m,n)\|\leq K' \left(\frac{\mu_m}{\mu_n}\right)^{a} \quad \text{ and } \quad 
\|\cB ^\mu(n,m)\|\leq K' \left(\frac{\mu_m}{\mu_n}\right)^{a} 
\end{equation*}
for $m\geq n$. Moreover, the following assertions are equivalent:
\begin{itemize}
    \item[i)] $(A_n)_{n\in \N}$ admits a $\mu$-dichotomy;
    \item[ii)] $(B^\mu_n)_{n\in \N}$ admits an exponential dichotomy
\end{itemize}
    
\end{theorem}

\begin{remark}
    It is easy to see that exponential, polynomial and logarithmic growth rates (recall Remark \ref{remark: particular cases of dich}), which are the most common growth rates explored in the literature, satisfy condition \eqref{eq: grbound}. 
\end{remark}

\subsection{$\mu$-dichotomy spectra} 
Let $\mu$ be a growth rate that satisfies \eqref{eq: grbound} and consider a sequence $\mathbb A=(A_n)_{n\in \N}$ of invertible linear operators on $\R^d$. Moreover, associated with $\mathbb A$, let us consider the sequence $\mathbb B=(B^\mu_n)_{n\in \N}$ given by \eqref{eq:B-mu}.

\begin{definition} \label{def: mu spect}
We define $\Sigma_{\mu D, \mathbb A}$ as the set of all $\tau \in \R$ with the property that the system
\begin{equation*}\label{lin1}
    x_{n+1}=\left (\frac{\mu_{n+1}}{\mu_n}\right )^{-\tau }A_n x_n   \quad n\in \N
\end{equation*}
does not admit a strong $\mu$-dichotomy. $\Sigma_{\mu D, \mathbb A}$ is called the \emph{strong $\mu$-dichotomy spectrum} of $\mathbb A$.
\end{definition}

\begin{remark}\label{aaa}
Similarly, we can introduce $\Sigma_{\mu D, \mathbb A}$ for two-sided sequences of invertible operators $(A_n)_{n\in \Z}$.
\end{remark}

An important particular case of strong $\mu$-dichotomy spectrum is the one given by the sequence $\mu_n=e^n$ for $n\in \N$. We are going to denote it by $\Sigma_{ED, \mathbb A}$ and call it the \emph{strong exponential dichotomy spectrum}

The following result shows that $\Sigma_{\mu D, \mathbb A}$ and $\Sigma_{ED, \mathbb B}$ are strongly connected.
\begin{theorem}\label{theo: DS-equal spect}
 Suppose that there exists $K,a>0$ such that $\mathbb A$ satisfies \eqref{bg}.
 Then
\[
\Sigma_{\mu D, \mathbb A}=\Sigma_{ED, \mathbb B}.
\]
\end{theorem}
\begin{proof}
    This result follows using Theorem \ref{theo: equiv exp mu dich} and proceeding as in the proof of \cite[Theorem 4.1]{DS}. Note that the difference between Theorem \ref{theo: DS-equal spect} and \cite[Theorem 4.1]{DS} relies on the fact that here we are considering a different notion of spectrum. More precisely, here we use the notion of \emph{strong} $\mu$-dichotomy instead of just $\mu$-dichotomy to define it.
\end{proof}

In particular, the previous result allows us to describe the structure of $\Sigma_{\mu D, \mathbb A}$. More precisely, suppose that there exists $K,a>0$ such that $\mathbb A$ satisfies \eqref{bg}. Then, by Theorem~\ref{theo: equiv exp mu dich} we find that $\mathbb B$ admits an exponential dichotomy and 
\begin{equation}\label{supp}
    \sup_{n\in \mathbb Z}\|B^\mu_n\|<+\infty \quad \text{and}\quad \sup_{n\in \mathbb Z}\|(B^\mu_n)^{-1}\|<+\infty. 
\end{equation}
Proceeding as in \cite[Theorem 3.4]{AS} or as in the classical work of Sacker and Sell \cite{SS-JDE78}, we get that there exist $1\le r\le d$ and $a_1\le b_1 <\ldots <a_r \le b_r$ such that
\begin{equation}\label{EDB}
\Sigma_{ED, \mathbb B}=\bigcup_{i=1}^r [a_i, b_i].
\end{equation}
Thus, by Theorem \ref{theo: DS-equal spect} we conclude that 
\begin{equation}\label{EDB-mu}
\Sigma_{\mu D, \mathbb A}=\bigcup_{i=1}^r [a_i, b_i].
\end{equation}

\section{Smooth linearization under exponential dichotomy} \label{sec: lin exp dich}

It is well known that  $C^0$ conjugacy  in general  does not preserve important dynamical properties such as  characteristic directions and  smoothness of invariant manifolds. Preserving these properties requires at least $C^1$ conjugacy. In the 1970s, an effort was made by Belitskii in~\cite{Bel73} to establish a $C^1$ linearization of $C^2$ hyperbolic diffeomorphisms on $\mathbb{R}^d$ under certain second-order non-resonance conditions. This result requires much weaker conditions on smoothness and non-resonance than the Sternberg's theorem in the $C^1$ case. Moreover, Hartman's example given in~\cite{HartPAMS60} shows that Belitskii's non-resonance conditions cannot be avoided.

Since the 1990s, two independent research directions emerged. In one direction, the  goal was to extend Belitskii's result to the infinite-dimensional setting, while the other direction was concerned with   differentiable   (at the fixed point $0$) linearization in the absence of non-resonance conditions. 
Concerning the first direction of research, we mention the work \cite{ZZJ}, which established  $C^1$ linearization result on  Banach spaces under appropriate  spectral gap and  spectral band conditions for the linear part. This result was further extended to the nonautonomous setting in \cite{DZZ}. In the second direction, van Strien~\cite{Stri-JDE90} claimed  the linearization result for  $C^2$ diffeomorphisms on $\R^d$ without any non-resonance conditions. The conclusion was that there exists  a conjugacy  which is simultaneously differentiable at 0 and H\"older
continuous near 0 (weaker than $C^1$ smoothness). However,  van Strien's proof was found to be incorrect (see~\cite{Ray-JDE98}). In~\cite{DZZ20}, the authors gave a correct proof of 
van Strien's result by using  different methods.

In what follows, we present two linearization results for nonautonomous dynamics under the assumption that the linear part admits an exponential dichotomy. The first result (see Theorem~\ref{t0}) gives conditions under which the conjugacies are $C^1$, while the second result (see Theorem~\ref{t1}) concerns the case when the conjugacies are differentiable at $0$ and locally H\"older
continuous. In contrast to the existing result in the literature, the novelty is that Theorems~\ref{t0} and~\ref{t1} are concerned with one-sided dynamics (which will make them applicable to the problem of linearization under $\mu$-dichotomic behavior).

\subsection{$C^1$ linearization under exponential dichotomy} 
We start by recalling that a strong exponential dichotomy is just a strong $\mu$-dichotomy, as described in Definition \ref{def: mu dichot}, with the growth rate $\mu$ given by $\mu_n=e^n$ for $n\in \Z^+$.

\begin{theorem}\label{t0}
Let $\mathbb B=(B_n)_{n\in \Z^+}$ be a sequence of invertible operators on $\R^d$ that admits a strong exponential dichotomy.
Suppose that $\Sigma_{ED, \mathbb B}$  has the form \eqref{EDB}, with
\begin{equation}\label{sp1}
a_1\le b_1 <\ldots <a_k\le b_k <0<a_{k+1}\le b_{k+1}<\ldots <a_r\le b_r,
\end{equation}
for some $1\leq k\leq r\leq d$ and
\begin{equation}\label{sp2}
\begin{cases}
a_{k+1}-b_k>\max \{b_r, -a_1\},\\
b_i-a_i\le -b_k \text{ for } 1\le i\le k,\\
b_i-a_i\le a_{k+1} \text{ for } k+1\le i \le r.\\
\end{cases}
\end{equation}
Moreover, let $f_n\colon \R^d \to \R^d$, $n\in \Z^+$, be a sequence of $C^1$ maps such that
\begin{equation}\label{fzero}
f_n(0)=0 \quad \text{and} \quad Df_n(0)=0,
\end{equation}
and
\begin{equation}\label{fm}
\|Df_n(x)v\| \le \eta \|v\| \; \text{ and } \; \|(Df_n(x)-Df_n(y))v\| \le L\|x-y\| \cdot \|v\|,
\end{equation}
with constants $L, \eta>0$ for $x, y, v\in \R^d$. Then, provided that $\eta$ is sufficiently small,  there exists a sequence of $C^1$-diffeomorphisms $h_n\colon \R^d\to \R^d$, $n\in \Z^+$,  such that
\begin{itemize}
\item for $n\in \Z^+$,
$
h_{n+1} \circ (B_n+f_n)=B_n\circ h_n;
$
\item there exist $M, \rho>0$ such that
\begin{equation*}
\|D h_n(x)v\|\le M\|v\| \; \text{ and }\; \|Dh_n^{-1}(x)v\| \le M\|v\|, \quad \forall n\in \Z^+,
\end{equation*}
for all $v\in \R^d$ and $x,y\in \mathbb{R}^d$ satisfying $\|x\|\le \rho$.
\end{itemize}
\end{theorem}

\begin{remark}
Theorem~\ref{t0} is a consequence of Theorem \ref{Ap} whose proof will be given in Appendix A. In fact, Theorem \ref{Ap} deals with the case of $n\in \Z$, while Theorem \ref{t0} deals with the case of $n\in \Z^+$. Moreover, observe that $0\notin \Sigma_{ED, \mathbb B}$ since $\mathbb B$ admits a strong exponential dichotomy.
\end{remark}

\begin{proof} 
Set $f_n:=0$ for all $n<0$. Then, \eqref{fm} holds for every $x, y\in \R^d$ and $n\in \Z$. Choose $c_1, \ldots, c_{r+1} \in \R$ such that
\[
c_1<a_1\le b_1<c_2 <a_2\le b_2 <\ldots < a_r \le b_r <c_{r+1}.
\]
For each $i\in \{1, 2, \ldots, r+1\}$, set
\[
S_i:=\bigg \{ x\in \R^d: \sup_{n\ge 0}\left ( \frac{1}{e^{c_in}} \| \mathcal B(n,0)x\|\right )<+\infty \bigg \}.
\]
Then, for each $i\in \{1,2, \ldots, r+1\}$, the sequence $(\frac{1}{e^{c_i}}B_n)_{n\in \Z^+}$ admits a strong exponential dichotomy with respect to projections $\tilde P_n^i$ for $n\in \Z^+$, where $\Ima \tilde P_0^i=S_i$ (see, for example, \cite[Proposition 3.3]{Bac25}). Moreover, the following holds true:
\begin{itemize}
\item $S_i$ does not depend on the particular choice of $c_i$ (e.g. \cite[Lemma 3.5]{Bac25});
\item $S_1=\{0\} \subsetneq S_2\ldots \subsetneq S_r \subsetneq S_{r+1}=\R^d$ (e.g. \cite[Lemma 3.6]{Bac25}).
\end{itemize}
Next, choose subspaces $V_2, \ldots, V_{r}$ of $\R^d$ such that
\[
\R^d=S_r\oplus V_r, \ S_r=S_{r-1} \oplus V_{r-1}, \ldots, S_3=S_2\oplus V_2.
\]
Hence,
\[
\R^d=V_r\oplus V_{r-1} \oplus \ldots \oplus V_2 \oplus S_2.
\]
Define an operator $B$ on $\R^d$ by
\begin{align}\label{BZ-}
B\rvert_{S_2}=e^{a_1}\Id_{S_2}, \ B\rvert_{V_2}=e^{a_2}\Id_{V_2}, \ldots, B\rvert_{V_r}=e^{a_r}\Id_{V_r},
\end{align}
where $\Id_V$ denotes the identity map on a subspace $V\subset \R^d$. Moreover, 
set $B_n:=B$ for $n<0$ and consider $\mathbb B':=(B_n)_{n\in \Z}$.

Let $\Sigma_{ED, \mathbb B'}$ be defined as in Definition \ref{def: mu spect} with $\mu_n=e^n$, $n\in \Z$ (recall Remark \ref{aaa}).
We now claim that $\Sigma_{ED, \mathbb B}=\Sigma_{ED, \mathbb B'}$. As mentioned above, $S_i$ does not depend on $c_i$. Thus, for any
\[
\tilde c_1 \in (-\infty, a_1), ~~\tilde c_i\in (b_{i-1}, a_{i}) \text{ for } i\in \{2,\ldots ,r\}, \text{ and }  \tilde c_{r+1}\in (b_r, \infty),
\]
we find that the sequence $(\frac {1}{e^{\tilde c_i}}B_n)_{n\in \Z^+}$ also admits a strong exponential dichotomy with projections $\tilde P_n^i$ for all $n\in \Z^+$ such that
$\Ima \tilde P_0^i=S_{i}$. Moreover, it is easy to see from \eqref{BZ-} that the sequence $\left(\frac {1}{e^{\tilde c_i}}B_n\right)_{n\in \Z^-}$ admits a strong exponential dichotomy on $\Z^-$ with projections $\tilde P_n^i$, $n\in \Z^-$ such that $\Ker \tilde P_0^-=V_{i} \oplus \ldots \oplus V_r$. Since
\[
\R^d=S_{i}\oplus V_{i} \oplus \ldots \oplus V_r,
\]
we conclude from the proof of~\cite[Theorem 2.3]{BDV0} that $(\frac {1}{e^{\tilde c_i}}B_n)_{n\in \Z}$ admits a strong exponential dichotomy on $\Z$.
This implies that $\R \setminus \Sigma_{ED, \mathbb B} \subset \R \setminus \Sigma_{ED, \mathbb B'}$, i.e., $\Sigma_{ED, \mathbb B'}\subset \Sigma_{ED, \mathbb B}$. On the other hand, we trivially have that $\Sigma_{ED, \mathbb B}\subset \Sigma_{ED, \mathbb B'}$, proving that the two spectra coincide, i.e.,
\[
\Sigma_{ED, \mathbb B'}
=\Sigma_{ED, \mathbb B}
=\bigcup_{i=1}^r [a_i, b_i].
\]
The conclusion of the theorem now follows readily from Theorem~\ref{Ap} given in Appendix A (which is a corollary of \cite[Theorem 2]{DZZ}).
\end{proof}

\subsection{Differentiable and H\"older linearization under exponential dichotomy}
Using similar arguments to the proof of Theorem \ref{t0}, we can also obtain the following result, where we remove the spectral gap condition, i.e., the first inequality of (\ref{sp2}), and obtain the simultaneously differentiable (at 0) and H\"older linearization.
\begin{theorem}\label{t1}
Suppose that $\mathbb B=(B_n)_{n\in \Z^+}$ is as in the statement  of~Theorem {\rm \ref{t0}} and that $\Sigma_{ED, \mathbb B}$ has the form \eqref{EDB} and satisfies \eqref{sp1} and
\begin{equation}\label{sp3}
b_i-a_i\le -b_k \text{ for } 1\le i\le k \;\text{ and }\; b_j-a_j\le a_{k+1} \text{ for } k+1\le j \le r.
\end{equation}
Let $f_n\colon \R^d \to \R^d$, $n\in \Z^+$, be a sequence of $C^1$ maps such that \eqref{fzero} and \eqref{fm} hold and let $\alpha_1\in \mathbb{R}$ be an arbitrary constant satisfying
\begin{equation*}\label{alpha1}
    0<\alpha_1<\min\Big\{\frac{ a_{k+1}- b_k}{ b_r},\frac{ a_{k+1}- b_{k}}{ -a_1}\Big\}.
\end{equation*}
Then, provided that $\eta$ is sufficiently small, there exists a sequence of homeomorphisms $h_n\colon \R^d\to \R^d$, $n\in \Z^+$, such that
\begin{itemize}
\item for $n\in \Z^+$,
$
h_{n+1} \circ (B_n+f_n)=B_n\circ h_n;
$
\item there exist constants $\tilde L, \rho>0$ and $0<\varrho<1$ such that
\begin{align*}
\|h_n(x)-x\|=o(\|x\|^{1+\varrho}) \; \text{ and } \; \|h_n^{-1}(x)-x\|=o(\|x\|^{1+\varrho})
\end{align*}
as $\|x\|\to 0$ and
\begin{align*}
\|h_n(x)-h_n(y)\|\!\le\! \tilde L\|x-y\|^{\alpha_1} \; \text{ and }\;
\|h_n^{-1}(x)-h_n^{-1}(y)\|\!\le\! \tilde L \|x-y\|^{\alpha_1}
\end{align*}
for $x,y\in \mathbb{R}^d$ satisfying $\|x\|,\|y\|\le \rho$.
\end{itemize}
\end{theorem}

\begin{proof}
The proof can be obtained in a similar manner to the proof of Theorem~\ref{t0}. More precisely, we extend the sequences $\mathbb B=(B_n)_{n\in \Z^+}$ and $(f_n)_{n\in \Z^+}$ to two-sided sequences $(B_n)_{n\in \Z}$ and $(f_n)_{n\in \Z}$
 exactly as in the proof of Theorem~\ref{t0}. Then, it remains to apply~\cite[Lemma 4]{DZZ20} to $\mathbb A^*$ (for $A_n:=B_n$, $n\in \mathbb Z)$ and $\mathbb F$ introduced in Appendix A (see the proof of Theorem~\ref{Ap} and Remark~\ref{855rem}).
\end{proof}

\section{Smooth linearization under $\mu$-dichotomy}
\label{sec: lin pol}

In this section, we present our main results concerning the linearization under $\mu$-dichotomic behavior.

\subsection{$C^1$ linearization under $\mu$-dichotomy}

\begin{theorem}\label{theo: main linearization discrete}
Let $\mu=(\mu_n)_{n\in \Z^+}$ be a growth rate that satisfies~\eqref{eq: grbound} for some $\theta \ge 1$ and $\mathbb A=(A_n)_{n\in \mathbb N}$ be a sequence of invertible linear operators on $\mathbb R^d$ that admits a strong $\mu$-dichotomy. Furthermore, suppose $\Sigma_{\mu D, \mathbb A}$ has the form \eqref{EDB-mu}
with $1\le r\le d$ and that the numbers $a_i$ and $b_i$ satisfy \eqref{sp1} and \eqref{sp2}.
Finally, let $(g_n)_{n\in \N}$ be a sequence of $C^1$ maps $g_n\colon \R^d \to \R^d$ satisfying the following:
\begin{itemize}
\item for every $n\in \N$,
\begin{equation}\label{eq: gn 0 and Dgn 0}
g_n(0)=0 \;\text{ and } \; Dg_n(0)=0;
\end{equation}
\item there exists $c>0$ such that 
\begin{equation}\label{eq: bound deriv g}
\|Dg_n(x)\| \le c\frac{\mu_n'}{\mu_n}, \quad \text{for $x\in \R^d$ and $n\in \N$,}
\end{equation}
where $\mu_n'=\mu_{n+1}-\mu_n$;
\item there exists $M>0$ such that 
\begin{equation}\label{4-4-3}
\|Dg_n(x)-Dg_n(y)\| \le M\frac{\mu_n'}{\mu_n}\|x-y\|, \quad \text{for $x, y\in \R^d$ and $n\in \N$.}
\end{equation}
\end{itemize}
Then, provided that $c$ is sufficiently small, there exists a sequence $(\psi_n)_{n\in \N}$ of $C^1$-diffeomorphisms $\psi_n\colon \R^d \to \R^d$ such that
\begin{equation}\label{lincond}
\psi_{n+1}\circ (A_n+g_n)=A_n  \circ \psi_n, \; \text{ for } n\in \N,
\end{equation}
and there exist $T, \rho>0$ so that 
\begin{equation}\label{hn-1}
\|D\psi_n(x)\| \le T \quad \text{and} \quad \|D\psi_n^{-1}(x)\| \le T,
\end{equation}
for $n\in \N$ and $x\in \R^d$ with $\|x\|\le \rho$.
\end{theorem}

\begin{proof} Given $n\in \N$, let $G_n:=A_n+g_n$. We start observing that, whenever $c>0$ given by \eqref{eq: bound deriv g} is small enough, $G_n$ is a bijection (a diffeomorphism, in fact) for every $n\in \N$. For this purpose, note that \eqref{eq: bound deriv g} implies that
\begin{equation}\label{LIP}
    \|g_n(x)-g_n(y)\|\leq c\frac{\mu_n'}{\mu_n}\|x-y\|
\end{equation}
for every $x,y\in \R^d$. Now, given $y\in \R ^d$, we observe that the map $x\mapsto A_n^{-1}y-A_n^{-1}g_n(x)$ is a contraction on $\R^d$. Indeed, given $x_1, x_2\in \R ^d$, using \eqref{bg}, \eqref{eq: grbound} and~\eqref{LIP}, we have that
\[
\begin{split}
\|A_n^{-1}y-A_n^{-1}g_n(x_1)-(A_n^{-1}y-A_n^{-1}g_n(x_2))\| &=\|A_n^{-1}g_n(x_1)-A_n^{-1}g_n(x_2)\| \\
&\le K \left (\frac{\mu_{n+1}}{\mu_n}\right )^a \|g_n(x_1)-g_n(x_2)\| \\
&\le cK\left (\frac{\mu_{n+1}}{\mu_n}\right )^a \frac{\mu_{n}'}{\mu_n}\|x_1-x_2\| \\
&\le cK\left (\frac{\mu_{n+1}}{\mu_n}\right )^{a+1}\|x_1-x_2\| \\
&\le cK\theta^{a+1}\|x_1-x_2\|.
\end{split}
\]
Therefore, provided that $cK\theta^{a+1}<1$, we see that
$x\mapsto A_n^{-1}y-A_n^{-1}g_n(x)$ is a contraction and, in particular, has a unique fixed point. That is, there exists a unique $x\in \R^d$ that satisfies $x=A_n^{-1}y-A_n^{-1}g_n(x)$, which is equivalent to $G_ n(x)=A_n x+g_n(x)=y$. This shows that $G_n$ is a bijection as claimed. In particular, we can consider
\begin{align*}
\mathcal G(m,n):=
\begin{cases}
G_{m-1} \circ \ldots \circ G_n, & m>n;\\
\Id, & m=n;\\
G_{m}^{-1} \circ \ldots \circ G_{n-1}^{-1}, & m<n.
\end{cases}
\end{align*}

Let $\mathbb B=(B_n)_{n\in \N}:=(B^\mu_n)_{n\in \N}$ be the sequence given by~\eqref{eq:B-mu}. By Theorem~\ref{theo: equiv exp mu dich}, we have that $\mathbb B$ admits an exponential dichotomy and satisfies \eqref{supp}.
Moreover, Theorem~\ref{theo: DS-equal spect} gives $\Sigma_{ED, \mathbb B}=\Sigma_{\mu D, \mathbb A}$. For $n\in \N$, let $f_n\colon \R^d\to \R^d$ be given by
\begin{equation}\label{fnx}
f_n(x):=\sum_{j=\lfloor \widetilde{\mu}^{-1}(e^{n-1})\rfloor+1}^{\lfloor \widetilde{\mu}^{-1}(e^{n})\rfloor}\cA(\lfloor \widetilde{\mu}^{-1}(e^{n})\rfloor+1, j+1)g_j (\mathcal G(j, \lfloor \widetilde{\mu}^{-1}(e^{n-1})\rfloor+1)(x)). 
\end{equation}
Our objective now is to show that $\mathbb B$ and $(f_n)_{n\in \N}$ satisfy the hypotheses of Theorem \ref{t0} so that we can translate the conclusions of that theorem into our setting. We start with some auxiliary observations.

Note that, by~\eqref{eq: grbound}, for each $t \in \Z^+$, we have
\[
\frac{\tilde{\mu}(t+1)}{\tilde{\mu}(t)}=\frac{\mu_{t+1}}{\mu_t}\le \theta
\]
and, for each $t \in [0,+\infty)\setminus\Z^+$, we have
\[
\frac{\tilde{\mu}(t+1)}{\tilde{\mu}(t)}=\frac{\mu_{r+1}+(t-r)(\mu_{r+2}-\mu_{r+1})}{\mu_r+(t-r)(\mu_{r+1}-\mu_r)}\le \frac{\mu_{r+2}}{\mu_r}=\frac{\mu_{r+2}}{\mu_{r+1}}\frac{\mu_{r+1}}{\mu_r} \le \theta^2,
\]
where $r=\lfloor t\rfloor \in \Z^+$. Consequently, since $\theta \ge 1$, we conclude that
\begin{equation}\label{eq:bound2}
\frac{\tilde{\mu}(t+1)}{\tilde{\mu}(t)}\le \theta^2, \quad \text{for all} \ t \in [0,+\infty).
\end{equation}
We also have the following useful estimate from \cite[Lemma 3.1]{DSV}: \begin{equation}\label{nm-1}
\sum_{j=n}^{m-1}\frac{\mu_j'}{\mu_j}\le \theta \log \left (\frac{\mu_m}{\mu_n}\right ).
\end{equation}

Next, we claim  that there exists $\tilde a \ge a$ such that 
\begin{equation}\label{krt}
\|D\mathcal G(m, n)(x)\| \le  K \left (\frac{\mu_{\max (m,n)}}{\mu_{\min (m,n)}}\right )^{\tilde a}, \quad \text{for $m, n\in \N$ and $x\in \mathbb R^d$}.
\end{equation}
Suppose initially that $m\ge n$. Observe that
\begin{equation*}
\mathcal G(m, n)(x)=\cA(m, n)x+\sum_{j=n}^{m-1} \cA (m, j+1)g_j(\mathcal G(j, n)(x)),
\end{equation*}
for $x\in \R^d$ and $m\ge n \ge 1$ and thus 
\begin{equation}\label{gj}
D\mathcal G(m, n)(x)=\cA(m, n)+\sum_{j=n}^{m-1}\cA(m, j+1)Dg_j(\mathcal G(j, n)(x))D\mathcal G(j, n)(x).
\end{equation}
Consequently, using \eqref{bg} and \eqref{eq: bound deriv g}, we get
\[
\begin{split}
\|D\mathcal G(m, n)(x)\| &\le K\left (\frac{\mu_m}{\mu_n}\right )^a+Kc\sum_{j=n}^{m-1}\left (\frac{\mu_m}{\mu_{j+1}}\right )^a \frac{\mu_j'}{\mu_j}\|D\mathcal G(j, n)(x)\| \\
&\le K\left (\frac{\mu_m}{\mu_n}\right )^a+Kc\sum_{j=n}^{m-1}\left (\frac{\mu_m}{\mu_{j}}\right )^a \frac{\mu_j'}{\mu_j}\|D\mathcal G(j, n)(x)\|.
\end{split}
\]
Hence, 
\begin{equation}\label{eq: aux Gronw}
\left (\frac{\mu_n}{\mu_m}\right )^a \|D\mathcal G(m, n)(x)\| \le K+Kc \sum_{j=n}^{m-1}\frac{\mu_j'}{\mu_j}\left (\frac{\mu_n}{\mu_j}\right )^a \|D\mathcal G(j, n)(x)\|.
\end{equation}
Now, in order to get the desired estimate, we need the following discrete-version of Gronwall's lemma (see, e.g. \cite[Lemma 4.32]{E-book}).
\begin{lemma}\label{lm-G}
Let $n\in \N$ and $\alpha >0$. Suppose that $(u_m)_{m\ge n}$ and $(z_m)_{m\ge n}$ are two nonnegative sequences that 
$
u_m \le K\{\alpha + \sum_{j=n}^{m-1} z_ju_j\}
$
for $m \ge n$.
Then
\begin{align*}
u_m \le K\alpha \,e^{K\sum_{j=n}^{m-1} z_j}, \quad \forall m \ge n.
\end{align*}
\end{lemma}
Thus, applying Gronwall's lemma to \eqref{eq: aux Gronw} and using \eqref{nm-1} we get that   
\[
\|D\mathcal G(m, n)(x)\| \le K\left (\frac{\mu_m}{\mu_n}\right )^a e^{Kc\sum_{j=n}^{m-1}\frac{\mu_j'}{\mu_j}} \le K\left (\frac{\mu_m}{\mu_n}\right )^{a+Kc\theta}.
\]
Therefore, \eqref{krt} holds with $\tilde a:=a+Kc\theta$. The case where $m<n$ can be obtained similarly.

As a final auxiliary result, we need to estimate
\[
\|D\mathcal G(m, n)(x)-D\mathcal G(m, n)(y)\|
\]
for $m\ge n$. By~\eqref{gj} we have that
\[
\begin{split}
&D\mathcal G(m, n)(x)-D\mathcal G(m, n)(y)\\
&=\sum_{j=n}^{m-1}\cA(m, j+1)Dg_j(\mathcal G(j, n)(x))D\mathcal G(j, n)(x) \\
&\phantom{=}-\sum_{j=n}^{m-1}\cA(m, j+1)Dg_j(\mathcal G(j, n)(y))D\mathcal G(j, n)(y) \\
&=\sum_{j=n}^{m-1}\cA(m, j+1)Dg_j(\mathcal G(j, n)(x))(D\mathcal G(j, n)(x)-D\mathcal G(j, n)(y))\\
&\phantom{=}-\sum_{j=n}^{m-1}\cA(m, j+1)(D g_j(\mathcal G(j, n)(y))-D g_j(\mathcal G(j, n)(x))D\mathcal G(j, n)(y).
\end{split}
\]
Thus, using \eqref{bg} and~\eqref{eq: bound deriv g}, we get that 
\[
\begin{split}
&\left \|\sum_{j=n}^{m-1}\cA(m, j+1)Dg_j(\mathcal G(j, n)(x))(D\mathcal G(j, n)(x)-D\mathcal G(j, n)(y))\right \| \\
&\le \sum_{j=n}^{m-1}\left \|\cA(m, j+1)Dg_j(\mathcal G(j, n)(x))(D\mathcal G(j, n)(x)-D\mathcal G(j, n)(y)) \right \| \\
&\le Kc \sum_{j=n}^{m-1}\left (\frac{\mu_m}{\mu_{j+1}}\right )^a  \frac{\mu_j'}{\mu_j}\|D\mathcal G(j, n)(x)-D\mathcal G(j, n)(y)\| \\
&\le Kc \sum_{j=n}^{m-1}\left (\frac{\mu_m}{\mu_{j}}\right )^a  \frac{\mu_j'}{\mu_j}\|D\mathcal G(j, n)(x)-D\mathcal G(j, n)(y)\|.
\end{split}
\]
Moreover, \eqref{bg}, \eqref{4-4-3}, \eqref{nm-1} and \eqref{krt} imply that 
\[
\begin{split}
&\left \|\sum_{j=n}^{m-1}\cA(m, j+1)(D g_j(\mathcal G(j, n)(y))-D g_j(\mathcal G(j, n)(x))D\mathcal G(j, n)(y)\right \| \\
&\le KM\sum_{j=n}^{m-1}\left (\frac{\mu_m}{\mu_{j+1}}\right )^a \frac{\mu_j'}{\mu_j}\|\mathcal G(j, n)(x)-\mathcal G(j, n)(y)\| \cdot \|D\mathcal G(j, n)(y)\| \\
&\le K^3M \|x-y\|\sum_{j=n}^{m-1}\left (\frac{\mu_m}{\mu_j}\right )^a \frac{\mu_j'}{\mu_j} \left (\frac{\mu_j}{\mu_n}\right )^{2\tilde a} \\
&\le K^3M\|x-y\|\left (\frac{\mu_m}{\mu_n}\right )^{2\tilde a}\left (\frac{\mu_m}{\mu_n}\right )^a\sum_{j=n}^{m-1}\frac{\mu_j'}{\mu_j} \\
&\le K^3M \theta \|x-y\|\left (\frac{\mu_m}{\mu_n}\right )^{2\tilde a} \left (\frac{\mu_m}{\mu_n}\right )^a \log \left (\frac{\mu_m}{\mu_n}\right ).
\end{split}
\]
Combining these observations, we conclude that, for $m\geq n$,
\[
\begin{split}
\left (\frac{\mu_n}{\mu_m}\right )^a\|D\mathcal G(m, n)(x)-D\mathcal G(m, n)(y)\| &\le K^3M \theta \|x-y\|\left (\frac{\mu_m}{\mu_n}\right )^{2\tilde a} \log \left (\frac{\mu_m}{\mu_n}\right ) \\
&\phantom{\le}+Kc \sum_{j=n}^{m-1}\left (\frac{\mu_n}{\mu_{j}}\right )^a  \frac{\mu_j'}{\mu_j}\|D\mathcal G(j, n)(x)-D\mathcal G(j, n)(y)\|.
\end{split}
\]
Thus, using Gronwall's lemma once again, we get that
\begin{equation}\label{eq: DG is holder}
\begin{split}
\|D\mathcal G(m, n)(x)-D\mathcal G(m, n)(y)\| &\le K^3M \theta \|x-y\|\left (\frac{\mu_m}{\mu_n}\right )^{3\tilde a+Kc\theta} \log \left (\frac{\mu_m}{\mu_n}\right )\\
\end{split}
\end{equation}
for $m\geq n$.

We now estimate the size of $\|Df_n(x)\|$. Note that
\[
\begin{split}
&Df_n(x)=\\
&\sum_{j=\lfloor \widetilde{\mu}^{-1}(e^{n-1})\rfloor+1}^{\lfloor \widetilde{\mu}^{-1}(e^{n})\rfloor}\cA(\lfloor \widetilde{\mu}^{-1}(e^{n})\rfloor+1, j+1)Dg_j (\mathcal G(j, \lfloor \widetilde{\mu}^{-1}(e^{n-1})\rfloor+1)(x))D\mathcal G(j, \lfloor \widetilde{\mu}^{-1}(e^{n-1})\rfloor+1)(x).
\end{split}
\]
Thus, using \eqref{bg}, \eqref{eq: bound deriv g}, \eqref{nm-1}, and~\eqref{krt}, we get that 
\begin{equation*}\label{dfnx}
\begin{split}
\|Df_n(x)\| &\le K^2 c\sum_{j=\lfloor \widetilde{\mu}^{-1}(e^{n-1})\rfloor+1}^{\lfloor \widetilde{\mu}^{-1}(e^{n})\rfloor}\left (\frac{\mu_{\lfloor \widetilde{\mu}^{-1}(e^{n})\rfloor+1}}{\mu_{j+1}}\right )^a  \frac{\mu_j'}{\mu_j} \left (\frac{\mu_j}{\mu_{\lfloor \widetilde{\mu}^{-1}(e^{n-1})\rfloor+1}}\right)^{\tilde a} \\
&\le K^2c\left (\frac{\mu_{\lfloor \widetilde{\mu}^{-1}(e^{n})\rfloor+1}}{\mu_{\lfloor \widetilde{\mu}^{-1}(e^{n-1})\rfloor+1}}\right )^{a+\tilde a}\sum_{j=\lfloor \widetilde{\mu}^{-1}(e^{n-1})\rfloor+1}^{\lfloor \widetilde{\mu}^{-1}(e^{n})\rfloor}\frac{\mu_j'}{\mu_j} \\
&\le \theta K^2c\left (\frac{\mu_{\lfloor \widetilde{\mu}^{-1}(e^{n})\rfloor+1}}{\mu_{\lfloor \widetilde{\mu}^{-1}(e^{n-1})\rfloor+1}}\right )^{a+\tilde a}\log \left (\frac{\mu_{\lfloor \widetilde{\mu}^{-1}(e^{n})\rfloor+1}}{\mu_{\lfloor \widetilde{\mu}^{-1}(e^{n-1})\rfloor+1}}\right ) \\
&\le c\theta^{1+2(a+\tilde a)}K^2 e^{a+\tilde a}\log (\theta^2 e)
\end{split}
\end{equation*} 
where in the last inequality, we have used that (recall the properties of $\widetilde{\mu}$ and $\widetilde{\mu}^{-1}$ given in Section \ref{sec: time rescaling} and \eqref{eq:bound2})
\begin{equation}\label{eq: aux mu mu tilde inverse}
\begin{split}
\frac{\mu_{\lfloor \widetilde{\mu}^{-1}(e^{n})\rfloor+1}}{\mu_{\lfloor \widetilde{\mu}^{-1}(e^{n-1})\rfloor+1}} \le \frac{\widetilde{\mu}(\widetilde{\mu}^{-1}(e^n)+1)}{e^{n-1}}=\frac{\widetilde{\mu}(\widetilde{\mu}^{-1}(e^n)+1)}{\widetilde{\mu}(\widetilde{\mu}^{-1}(e^n))} \ \frac{e^{n}}{e^{n-1}}\le e\theta^2.
\end{split}
\end{equation}
Hence, there exists $C>0$ such that 
\begin{equation*}\label{8-3-7}
    \|Df_n(x)\| \le cC, \quad \text{for $n\in \N$ and $x\in \R^d$.}
\end{equation*}

Finally, recalling \eqref{fnx}, we have that 
\begin{equation}\label{eq: exp fn Gn An}
f_n=\mathcal G(\lfloor \widetilde{\mu}^{-1}(e^{n})\rfloor+1, \lfloor \widetilde{\mu}^{-1}(e^{n-1})\rfloor+1)-\cA(\lfloor \widetilde{\mu}^{-1}(e^{n})\rfloor+1, \lfloor \widetilde{\mu}^{-1}(e^{n-1})\rfloor+1).
\end{equation}
Thus,
\begin{equation*}\label{Dfnxy}
\begin{split}
&Df_n(x)-Df_n(y)\\
&=D\mathcal G(\lfloor \widetilde{\mu}^{-1}(e^{n})\rfloor+1, \lfloor \widetilde{\mu}^{-1}(e^{n-1})\rfloor+1)(x)-D\mathcal G(\lfloor \widetilde{\mu}^{-1}(e^{n})\rfloor+1, \lfloor \widetilde{\mu}^{-1}(e^{n-1})\rfloor+1)(y).
\end{split}
\end{equation*}
Consequently, by \eqref{eq: DG is holder} and \eqref{eq: aux mu mu tilde inverse}, there exists $\tilde M>0$ such that 
\[
\|Df_n(x)-Df_n(y)\| \le \tilde M\|x-y\|, \quad \text{for $n\in \N$ and $x, y\in \R^d$.}
\]

Combining the previous observations, we conclude that $\mathbb B$ and $(f_n)_{n\in \N}$ satisfy the hypotheses of Theorem \ref{t0}. Thus, there is a sequence $(h_n)_{n\in \N}$ of $C^1$-diffeomorphisms $h_n\colon \R^d \to \R^d$ satisfying
\begin{equation}\label{LIN-proof}
    h_{n+1}\circ (B_n+f_n)=B_n\circ h_n
\end{equation}
for $n\in \N$. Moreover, there exist $T_1, \rho_1>0$ such that 
\begin{equation}\label{psi-1}
\|Dh_n(x)\| \le T_1 \quad \text{and} \quad \|Dh_n^{-1}(x)\| \le T_1,
\end{equation}
for $n\in \N$ and $x\in \R^d$ with $\|x\|\le \rho_1$.
In particular, recalling \eqref{eq:B-mu} and \eqref{eq: exp fn Gn An}, we get that
\begin{equation}\label{lin}
    h_{n+1}\circ \mathcal G(\lfloor \widetilde{\mu}^{-1}(e^{n})\rfloor+1, \lfloor \widetilde{\mu}^{-1}(e^{n-1})\rfloor+1)=\cA (\lfloor \widetilde{\mu}^{-1}(e^{n})\rfloor+1, \lfloor \widetilde{\mu}^{-1}(e^{n-1})\rfloor+1)\circ h_n,
\end{equation}
for $n \in \mathbb N$.

Now, given $k\in \N$, choose $n\in \Z^+$ such that 
\begin{equation}\label{nn+1}
\lfloor\widetilde\mu^{-1}(e^{n+1})\rfloor +1>  k
\ge \lfloor\widetilde\mu^{-1}(e^n)\rfloor+1
\end{equation}
and define $\psi_k\colon \R^d \to \R^d$ by 
\begin{equation}\label{psik}
\psi_k:=\cA (k, \lfloor\widetilde\mu^{-1}(e^{n})\rfloor+1)\circ h_{n+1}\circ \mathcal G(\lfloor\widetilde\mu^{-1}(e^{n})\rfloor+1, k).
\end{equation}
Clearly, $\psi_k$ is a $C^1$-diffeomorphism for each $k\in \N$.
We claim that 
\begin{equation}\label{linearization}
    \psi_{k+1}\circ (A_k+g_k)=A_k \circ \psi_k, \quad \text{for every $k\in \N$.}
\end{equation}
Indeed, fix $k\in \N$ and choose $n\in \Z^+$ so that~\eqref{nn+1} holds. If $\lfloor\widetilde\mu^{-1}(e^{n+1})\rfloor +1>k+1$, then 
\[
\begin{split}
\psi_{k+1}\circ (A_k+g_k) &=\cA (k+1, \lfloor\widetilde\mu^{-1}(e^{n})\rfloor+1)\circ h_{n+1}\circ \mathcal G(\lfloor\widetilde\mu^{-1}(e^{n})\rfloor+1, k+1)\circ \mathcal G(k+1, k)\\
&=A_k \circ \cA (k, \lfloor\widetilde\mu^{-1}(e^{n})\rfloor+1)\circ h_{n+1}\circ \mathcal G(\lfloor\widetilde\mu^{-1}(e^{n})\rfloor+1, k) \\
&=A_k \circ \psi_k,
\end{split}
\]
yielding~\eqref{linearization}.  On the other hand, if $\lfloor\widetilde\mu^{-1}(e^{n+1})\rfloor +1=k+1$, then using~\eqref{lin} we get that 
\[
\begin{split}
\psi_{k+1}\circ (A_k+g_k) &=h_{n+2} \circ \mathcal G(\lfloor\widetilde\mu^{-1}(e^{n+1})\rfloor +1, \lfloor\widetilde\mu^{-1}(e^{n+1})\rfloor) \\
&=h_{n+2}\circ \mathcal G(\lfloor\widetilde\mu^{-1}(e^{n+1})\rfloor +1, \lfloor\widetilde\mu^{-1}(e^{n})\rfloor +1)\circ \mathcal G(\lfloor\widetilde\mu^{-1}(e^{n})\rfloor +1, \lfloor\widetilde\mu^{-1}(e^{n+1})\rfloor)\\
&=\cA (\lfloor\widetilde\mu^{-1}(e^{n+1})\rfloor +1, \lfloor\widetilde\mu^{-1}(e^{n})\rfloor +1)\circ h_{n+1}\circ \mathcal G(\lfloor\widetilde\mu^{-1}(e^{n})\rfloor +1, \lfloor\widetilde\mu^{-1}(e^{n+1})\rfloor)\\
&=A_k\circ \cA (\lfloor\widetilde\mu^{-1}(e^{n+1})\rfloor , \lfloor\widetilde\mu^{-1}(e^{n})\rfloor +1)\circ h_{n+1}\circ \mathcal G(\lfloor\widetilde\mu^{-1}(e^{n})\rfloor +1, \lfloor\widetilde\mu^{-1}(e^{n+1})\rfloor)\\
&=A_k\circ \psi_k.
\end{split}
\]
Thus,~\eqref{linearization} also holds in this case.

Let us now estimate the size of $\|D\psi_k(x)\|$. For this purpose, we start by observing that, since $\mathcal G(m, n)(0)=0$, condition \eqref{krt} together with the mean-value theorem implies that 
\begin{equation}\label{gmnx}
\|\mathcal G(m, n)(x)\| \le K\left (\frac{\mu_{\max (m,n)}}{\mu_{\min (m,n)}}\right )^{\tilde a}\|x\|, \quad \text{for $m,n\in \N$ and $x\in \mathbb R^d$.}
\end{equation}
Thus, \eqref{bg}, \eqref{krt}, \eqref{eq: aux mu mu tilde inverse} and~\eqref{psi-1} gives us that
\[
\begin{split}
\|D\psi_k(x)\| &\le T_1K\left (\frac{\mu_k}{\mu_{\lfloor \widetilde{\mu}^{-1}(e^n)\rfloor+1}}\right )^a  \|D\mathcal G(\lfloor \widetilde{\mu}^{-1}(e^n)\rfloor+1, k)(x)\| \\
&\le T_1K^2 \left (\frac{\mu_k}{\mu_{\lfloor \widetilde{\mu}^{-1}(e^n)\rfloor+1}}\right )^{a+\tilde a} \\
&\le T_1 K^2 (e\theta^2)^{a+\tilde a},
\end{split}
\]
for $k\in \N$ and $x\in \R^d$ such that $\|x\| \le \frac{\rho_1}{K(e\theta^2)^{\tilde a}}$, so that (recall \eqref{eq: aux mu mu tilde inverse} and \eqref{gmnx})
\[
\|\mathcal G(\lfloor \widetilde{\mu}^{-1}(e^n)\rfloor+1, k)(x)\| \le K \left (\frac{\mu_k}{\mu_{\lfloor \widetilde{\mu}^{-1}(\eta_n)\rfloor+1}}\right )^{\tilde a}\|x\| \le K(e\theta^2)^{\tilde a}\|x\| \le \rho_1.
\]
This establishes the first estimate in~\eqref{hn-1}. The second can be established in a similar manner. This completes the proof of the theorem.
\end{proof}

We now apply Theorem \ref{theo: main linearization discrete} to a very simple example in the case where the growth rate $\mu$ is \emph{polynomial}, that is, $\mu_n=1+n$ for $n\in \N$, which is not covered by previously known results.
\begin{example}\label{example: main theo discrete}
Let $\mu=(\mu_n)_{n\in \N}$ be given by $\mu_n=1+n$ for $n\in \N$ and consider the sequences of linear operators $\mathbb{A}=(A_n)_{n\in \N}$ and $(P_n)_{n\in \N}$ acting on $\R^2$ given by
\[
A_n:=\begin{pmatrix}
\frac{n}{n+1} & 0\\
0 & \frac{n+1}{n}
\end{pmatrix} \quad\text{ and }\quad P_n:=\begin{pmatrix}
1 & 0\\
0 & 0
\end{pmatrix},\quad n\in \N.
\]
Then,
\[
\cA(m,n)=\begin{pmatrix}
\frac{n}{m} & 0\\
0 & \frac{m}{n}
\end{pmatrix} \quad  \text{for every } m,n\in \N.
\]
It is easy to see that $(A_n)_{n\in\N}$ admits a strong $\mu$-dichotomy with constants $K=a=\lambda=1$. Moreover, $\Sigma_{\mu D,\mathbb{A}}=\{-1,1\}$. In particular, $r=2$, $a_1=b_1=-1$ and $a_2=b_2=1$ and conditions \eqref{sp2} are satisfied.

Consider now $\xi\colon \R\to \R$ given by $\xi(x)=x^2e^{-x^2}$. Then, $D\xi(x)=2xe^{-x^2}(1-x^2)$ and, consequently, $|D\xi(x)|\leq 1$ and $|D\xi(x)-D\xi(y)|\leq 2|x-y|$. Thus, taking $g_n:\R^2\to \R^2$ as
\[g_n(x_1,x_2)=\frac{c}{n+1}(\xi(x_1),\xi(x_2)), \quad n\in \N,\]
where $c>0$ is a constant, it follows that conditions \eqref{eq: gn 0 and Dgn 0}, \eqref{eq: bound deriv g} and \eqref{4-4-3} are satisfied. In particular, Theorem \ref{theo: main linearization discrete} may be applied whenever $c>0$ is small enough. On the other hand, it is easy to see that the sequence  $\mathbb{A}=(A_n)_{n\in \N}$ does not admit a (nonuniform) strong exponential dichotomy. Therefore, previously known results such as the one in \cite{DZZ} can not be applied to this example.
\end{example}

The previous example shows how easy it is to construct nonautonomous systems with dichotomic behavior to which our result applies, and which were not covered by previous works. In particular, it clearly demonstrates the broad scope of Theorem \ref{theo: main linearization discrete}.

\subsection{Differentiable and H\"older linearization  under $\mu$-dichotomy}
In the next result, we remove the spectral gap condition from our hypotheses (i.e., the first inequality of \eqref{sp2}), and obtain a linearization that is simultaneously differentiable (at 0) and H\"older continuous in a neighborhood of 0.

\begin{theorem}\label{theo: diff+holder linearization}
Suppose that $\mathbb A=(A_n)_{n\in \N}$ and $\Sigma_{\mu D, \mathbb A}$ are given as in Theorem~\ref{theo: main linearization discrete} and that \eqref{sp1} and \eqref{sp3} hold.
Let $(g_n)_{n\in \N}$ be a sequence of $C^1$ maps $g_n\colon \R^d \to \R^d$ such that \eqref{eq: gn 0 and Dgn 0}, \eqref{eq: bound deriv g} and \eqref{4-4-3} hold, and let $\alpha_1$ and $\varrho$ be given in Theorem \ref{t1}.
Then, provided that $c$ is sufficiently small, there exists a sequence of homeomorphisms $\psi_n\colon \mathbb{R}^d\to \mathbb{R}^d$, $n\in \N$, such that \eqref{lincond} holds.
Moreover, there exist constants $ L', \rho'>0$ such that 
\begin{equation}\label{000}
\psi_n(x)=x+ o(\|x\|^{1+\varrho})\; \text{ and }\;
\psi_n^{-1}(x)=x+ o(\|x\|^{1+\varrho})
\end{equation}
as $\|x\|\to 0$ and
\begin{equation}\label{001}
\|\psi_n(x)-\psi_n(y)\|\!\le\!  L'\|x-y\|^{\alpha_1} \; \text{ and }\;
\|\psi_n^{-1}(x)-\psi_n^{-1}(y)\|\!\le\! \ L'  \|x-y\|^{\alpha_1}
\end{equation}
for $x,y\in \mathbb{R}^d$ satisfying $\|x\|,\|y\|\le \rho'$. 
\end{theorem}

\begin{proof}
We use the same notation as in the proof of Theorem~\ref{theo: main linearization discrete}. In particular, let $\mathbb B$ and $(f_n)_{n\in \N}$ be as in that proof. By Theorem~\ref{t1}, there exist a sequence of homeomorphisms $h_n\colon \R^d \to \R^d$, $n\in \Z^+$ satisfying~\eqref{LIN-proof} and  constants $\tilde L, \varrho,\rho>0$  such that 
\begin{equation}\label{cd0}
\|h_n(x)-x\|=o(\|x\|^{1+\varrho}) \;\text{ and }\;
\|h_n^{-1}(x)-x\|=o(\|x\|^{1+\varrho})
\end{equation}
as $\|x\|\to 0$ and
\begin{equation}\label{cd1}
\|h_n(x)-h_n(y)\|\!\le\! \tilde L\|x-y\|^{\alpha_1} \;\text{ and }\;
\|h_n^{-1}(x)-h_n^{-1}(y)\|\!\le\! \tilde L \|x-y\|^{\alpha_1},
\end{equation}
for $x,y\in\mathbb{R}^d$ satisfying $\|x\|,\|y\|\le \rho$. 

Next, we may construct a sequence of homeomorphisms $\psi_k \colon \R^d \to \R^d$, $k\in \N$, exactly as in the proof of Theorem~\ref{theo: main linearization discrete} (see~\eqref{psik}) such that~\eqref{lincond} holds. Now, observe that for
$\|x\| \le \frac{\rho}{K(e\theta^2)^{\tilde a}}$, we have that (recall \eqref{eq: aux mu mu tilde inverse} and \eqref{gmnx})
\[
\|\mathcal G(\lfloor \widetilde{\mu}^{-1}(e^n)\rfloor+1, k)(x)\| \le K \left (\frac{\mu_k}{\mu_{\lfloor \widetilde{\mu}^{-1}(\eta_n)\rfloor+1}}\right )^{\tilde a}\|x\| \le K(e\theta^2)^{\tilde a}\|x\| \le \rho.
\]
Therefore, taking $\|x\|,\|y\| \le \frac{\rho}{K(e\theta^2)^{\tilde a}}$ and $n$ as in \eqref{nn+1}, and using \eqref{bg},  \eqref{krt}, \eqref{eq: aux mu mu tilde inverse} and \eqref{cd1}, we have
{\small
\[
\begin{split}
&\| \psi_k(x)-\psi_k(y)\|\\
&=\|\cA (k, \lfloor\widetilde\mu^{-1}(e^{n})\rfloor+1)(h_{n+1}(\mathcal G(\lfloor\widetilde\mu^{-1}(e^{n})\rfloor+1, k)(x)))-\cA (k, \lfloor\widetilde\mu^{-1}(e^{n})\rfloor+1)( h_{n+1} ( \mathcal G(\lfloor\widetilde\mu^{-1}(e^{n})\rfloor+1, k)(y)))\| \\
&\le K\left (\frac{\mu_k}{\mu_{\lfloor \widetilde{\mu}^{-1}(e^n)\rfloor+1}}\right )^a\|h_{n+1}(\mathcal G(\lfloor\widetilde\mu^{-1}(e^{n})\rfloor+1, k)(x))- h_{n+1} ( \mathcal G(\lfloor\widetilde\mu^{-1}(e^{n})\rfloor+1, k)(y))\|\\
&\le \tilde{L} K\left (\frac{\mu_k}{\mu_{\lfloor \widetilde{\mu}^{-1}(e^n)\rfloor+1}}\right )^a\|\mathcal G(\lfloor\widetilde\mu^{-1}(e^{n})\rfloor+1, k)(x)-  \mathcal G(\lfloor\widetilde\mu^{-1}(e^{n})\rfloor+1, k)(y)\|^{\alpha_1}\\
&\le \tilde{L} K^{1+\alpha_1}\left (\frac{\mu_k}{\mu_{\lfloor \widetilde{\mu}^{-1}(e^n)\rfloor+1}}\right )^{a+\alpha_1 \tilde a}\| x-y\|^{\alpha_1}\\
&\le \tilde{L} K^{1+\alpha_1}(e\theta^2)^{a+\alpha_1 \tilde a}\| x-y\|^{\alpha_1}.\\
\end{split}
\]}
Hence, the first estimate in~\eqref{001} is proved by setting $\rho':=\frac{\rho}{K(e\theta^2)^{\tilde a}}$. Similarly, we can establish the second one.

We now claim that for each $k\in \N$,
\begin{equation}\label{cd0-k}
    \|\psi_k(x)-x\|=o(\|x\|^{1+\varrho})\quad \mbox{as $\|x\|\to 0$}.
\end{equation}
When $k=\lfloor\widetilde\mu^{-1}(e^n)\rfloor+1$ for some $n\in \N$, \eqref{cd0-k} follows readily from~\eqref{cd0} since $\psi_k=h_{n+1}$. Assume now that~\eqref{cd0-k} 
holds for some $\lfloor\widetilde\mu^{-1}(e^{n+1})\rfloor +1>  k
\ge \lfloor\widetilde\mu^{-1}(e^n)\rfloor+1
$. By~\eqref{linearization}, we have $\psi_{k+1}=A_k\circ \psi_k \circ (A_k+g_k)^{-1}$. Hence,  
\[
\begin{split}
\psi_{k+1}-\Id &=A_k\circ \psi_k \circ (A_k+g_k)^{-1}-\Id \\
&=A_k\circ (\Id+{\psi}_k-\Id)\circ (A_k^{-1}+(A_k+g_k)^{-1}-A_k^{-1})-\Id \\
&=A_k \circ ((A_k+g_k)^{-1}-A_k^{-1})+A_k\circ ({\psi}_k-\Id)\circ (A_k+g_k)^{-1}.
\end{split}
\]
Now, we estimate each of the summands in the last equality. First, note that 
\[A_k \circ ((A_k+g_k)^{-1}-A_k^{-1})(x)= -g_k\circ(A_k+g_k)^{-1}(x). \]
Moreover, since $Dg_k(0)=0$, it follows by \eqref{eq: grbound} and \eqref{4-4-3} that
\[\|Dg_k(x)\|\leq M \frac{\mu_k'}{\mu_k}\|x\|\leq M\theta\|x\| \; \text{ for } x\in \R^d. \]
Thus, since $g_k(0)=0$, combining the previous observation with the mean-value theorem, we get that
\[\|g_k(x)\|\leq M\theta \|x\|^2 \; \text{ for } x\in \R^d.\]
Putting all this together with \eqref{gmnx}, we get
\begin{align*}
\|A_k \circ ((A_k+g_k)^{-1}-A_k^{-1})(x)\|&= \|g_k\circ(A_k+g_k)^{-1}(x)\| \\
&\le M\theta \|(A_k+g_k)^{-1}(x)\|^2\\
&=o(\|x\|^{1+\varrho})
\end{align*}
as $\|x\|\to 0$ (recall that $\varrho <1$). Moreover, from \eqref{bg}, \eqref{eq: grbound}, \eqref{gmnx} and our assumption that \eqref{cd0-k} holds for $k$, we see that
\[
\begin{split}
&\|(A_k\circ ({\psi}_k-\Id)\circ (A_k+g_k)^{-1})(x)\|
\\
&\le  K\left(\frac{\mu_{k+1}}{\mu_k}\right) ^a\|(({\psi}_k-\Id)\circ (A_k+g_k)^{-1})(x)\|
\\
&= o(\|(A_k+g_k)^{-1}(x)\|^{1+\varrho})
=o(\|x\|^{1+\varrho})
\end{split}
\]
as $\|x\|\to 0$. From the last two estimates, we conclude that~\eqref{cd0-k} holds for $k+1$. By induction, we find that~\eqref{cd0-k} holds for each $k\in \N$, which proves the first estimate of~\eqref{000}. The second one can be proved similarly, and the proof of the theorem is completed. 
\end{proof}

\subsection{The case of infinite-dimensional dynamics}\label{sec: inf dim discrete} 
The purpose of this short subsection is to indicate how one can extend our results to the infinite-dimensional setting. Let $X$ be an arbitrary Hilbert space and $\mathcal B(X)$ be the space of all bounded linear operators on $X$.  

The notions of $\mu$-dichotomy and a strong $\mu$-dichotomy spectrum can be introduced for sequences of invertible operators in $\mathcal B(X)$ as in Section \ref{sec: pol dich}. Moreover, one can establish versions of Theorems \ref{theo: equiv exp mu dich} and \ref{theo: DS-equal spect} in an infinite-dimensional setting by arguing exactly as in the finite-dimensional setting. Indeed, note that in the proofs of these results, finite dimensionality was never used.
However, in the infinite-dimensional setting, $\Sigma_{ED, \mathbb A}$ may not have the form given in \eqref{EDB}. Hence, one can formulate versions of Theorems~\ref{theo: main linearization discrete}
and \ref{theo: diff+holder linearization} in the infinite-dimensional case by assuming an additional condition that $\Sigma_{\mu D, \mathbb A}$ has the form~\eqref{EDB-mu} (which in the finite-dimensional setting is automatically satisfied).

Finally, we note that the preparatory results Theorem~\ref{t0} and Theorem~\ref{Ap} can be established in the same manner. We only emphasize that the assumption that $X$ is a Hilbert space would be used in the construction of (closed) subspaces $S_i$ appearing in the proof of Theorem~\ref{t0}.

\section{The case of continuous time}\label{sec: continuous time case}
In this section, we present versions of Theorems \ref{theo: main linearization discrete} and \ref{theo: diff+holder linearization} in the case of continuous time dynamical systems. 

Let us consider a linear nonautonomous equation 
\begin{equation}\label{LDE}
x'=A(t)x, \quad t\ge 1,
\end{equation}
where $A:[1,\infty)\to \R^{d\times d}$ is a continuous map acting on $[1, \infty)$ and with values in the family of linear operators on $\R^d$. By $T(t,s)$ we denote the evolution family associated with~\eqref{LDE}.
 Moreover, given a continuous function  $f\colon [1, \infty) \times \R^d \to \R^d$, we consider the following semilinear differential equation
\begin{equation}\label{NDE}
x'=A(t)x+f(t,x), \quad t\ge 1.
\end{equation}

\subsection{Growth rates and $\mu$-dichotomies for continuous time dynamics} We say that a function $\mu\colon [1,+\infty)\to (0,+\infty)$ is a \emph{growth rate} if it is strictly increasing and $\lim_{t\to +\infty}\mu(t)=+\infty$. Moreover, if $\mu$ is differentiable, we say that it is a \emph{differentiable growth rate}.

\begin{definition}\label{def: nonunif pol dich continuous}
We say that~\eqref{LDE} (or equivalently, that the evolution family $(T(t,s))_{t,s\geq 1}$) admits a \emph{$\mu$-dichotomy} if there exist $K,\lambda>0$ and a family of projections $P(t)$, $t\ge 1$ on $\R^d$ such that the following properties hold:
\begin{itemize}
\item  for $t\ge s\ge 1$,
\begin{equation}\label{Pro}
P(t)T(t,s)=T(t,s)P(s);
\end{equation}
\item for $t\ge  s$,
\begin{equation}\label{B1}
\| T(t,s)P(s)\| \le K \bigg (\frac{\mu(t)}{\mu(s)}\bigg )^{-\lambda},
\end{equation}
and 
\begin{equation}\label{B2}
\|T(s,t)Q(t)\| \le K \bigg (\frac{\mu(t)}{\mu(s)}\bigg )^{-\lambda},
\end{equation}
where $Q(t)=\Id-P(t)$.
\end{itemize}
In addition, if there exists $a\geq \lambda$ such that for $t\ge s$,
\begin{equation}\label{bnd}
\| T(t,s)\| \le K\left(\frac{\mu(t)}{\mu(s)} \right)^a \quad \text{and} \quad \| T(s,t)\| \le K \left(\frac{\mu(t)}{\mu(s)}\right)^a,
\end{equation}
we say that \eqref{LDE} admits a \emph{strong $\mu$-dichotomy}.
\end{definition}
Analogously to the discrete time case, by taking $\mu(t)=e^t$, $t\geq 1$, we recover the classical notion of \emph{(strong) exponential dichotomy}; by taking $\mu(t)=1+t$, $t\geq 1$, we recover the notion of \emph{(strong) polynomial dichotomy}; and finally, by taking $\mu(t)=\ln(e+t)$, $t\geq 1$, we recover the notion of \emph{(strong) logarithmic dichotomy}.

In the next proposition, we present a relationship between the dynamics of \eqref{LDE} and a suitable discretization of it. 
\begin{proposition} \label{prop: equiv pol dich contXdisc}
Assume that there exist $K,a>0$ such that the evolution family $T(t,s)$ of \eqref{LDE} satisfies \eqref{bnd}. Moreover, suppose that the growth rate $\mu$ satisfies
\begin{equation}\label{eq: grbound-cont}
	\frac{\mu(t+1)}{\mu (t)}\le \theta, \quad \text{for all} \quad t\geq 1,
\end{equation} 
for some $\theta\ge 1$
Then, the following properties are equivalent:
\begin{itemize}
\item \eqref{LDE} admits a strong $\mu$-dichotomy;
\item the sequence $\mathbb A =(A_n)_{n\in \N}$ admits a strong $\mu$-dichotomy, where
\begin{equation}\label{AN}
A_n=T(n+1, n), \quad n\in \N.
\end{equation}
\end{itemize}
\end{proposition}

\begin{proof}
Taking $\cA(m,n)$ as in \eqref{eq: cocycle} with $(A_n)_{n\in \N}$ given by \eqref{AN}, we see that
\begin{equation}\label{AT}
\cA(m,n)=T(m,n), \quad m, n\in \N.
\end{equation}
Assume that~\eqref{LDE} admits a strong $\mu$-dichotomy with projections $P(t)$, $t\ge 1$. By~\eqref{B1} and~\eqref{AT}, we have
\[
\| \cA(m,n)P(n)x\| \le K \bigg (\frac{\mu(m)}{\mu (n)}\bigg )^{-\lambda}\|x\|, \quad m\ge n \text{ and }x\in\R^d.
\]
Similarly, \eqref{B2} and~\eqref{AT} imply that 
\[
\| \cA(n,m)Q(m)x\| \le K \bigg (\frac{\mu(m)}{\mu (n)}\bigg )^{-\lambda}\|x\|, \quad m\ge n \text{ and }x\in\R^d.
\]
Moreover, by \eqref{bnd} and~\eqref{AT}, we get that for every $x\in \R^d$ and $m\ge n$,
\begin{equation*}
\| \cA(m,n)x\| \le {K}\bigg (\frac{\mu (m)}{\mu(n)}\bigg )^{a}\|x\| \quad \text{and} \quad \| \cA(n,m)x\| \le {K}\bigg (\frac{\mu (m)}{\mu(n)}\bigg )^{a}\|x\|.
\end{equation*}
Consequently, we conclude that the sequence $(A_n)_{n\in \N}$ admits a strong $\mu$-dichotomy with projections $P(n)$, $n\in \N$.

Assume now that the sequence $(A_n)_{n\in \N}$ admits a strong $\mu$-dichotomy with projections $P_n$, $n\in \N$. Hence, there exist $K>0$ and $a\geq \lambda >0$ such that~\eqref{pro}, \eqref{pd1} and~\eqref{bg} hold.
For $n\in \N$ and $t\in [n, n+1)$, we set
\[
P(t)=T(t,n)P_n T(n,t).
\]
One can easily verify that $P(t)$ is a projection for each $t\ge 1$.
Take $t,s\geq 1$ and assume $t\geq s$. We choose $m, n\in \N$ such that $m\le t<m+1$ and $n\le s <n+1$. Clearly, $m\ge n$. Then
\[
\begin{split}
T(t,s)P(s)&=T(t,s)T(s,n)P_nT(n, s)\\
&=T(t,n)P_n T(n,s) =T(t,m)\cA(m,n)P_n T(n,s)\\
&=T(t,m)P_m\cA(m,n) T(n,s)=T(t,m)P_m T(m,t)T(t,s)\\
&=P(t)T(t,s).
\end{split}
\]
In particular, \eqref{Pro} holds. 

Moreover, the previous calculations combined with \eqref{pd1}, \eqref{bnd}, and \eqref{eq: grbound-cont} also show that for $x\in \R^d$,
\[
\begin{split}
\|T(t,s)P(s)x \| &=\| T(t,m)\cA(m,n)P_n T(n,s)x\|\\
&\leq K \left(\frac{\mu(t)}{\mu(m)}\right)^a K\left(\frac{\mu(m)}{\mu(n)}\right)^{-\lambda} K\left(\frac{\mu(s)}{\mu(n)}\right)^a\|x\|    \\
&\le K^3 \theta^{2a} \bigg (\frac{\mu(m)}{\mu(n)}\bigg )^{-\lambda}\|x\|
\le K^3 \theta^{2a+\lambda} \bigg (\frac{\mu(t)}{\mu (s)} \bigg )^{-\lambda}\|x\|,
\\
\end{split}
\]
which yields~\eqref{B1}. Similarly, one can establish~\eqref{B2}. Finally, noting that \eqref{bnd} is satisfied by our hypotheses, it follows that \eqref{LDE} admits a strong $\mu$-dichotomy as claimed. This completes the proof of the proposition.
\end{proof}

\subsection{$\mu$-dichotomy spectrum} Let $\mu\colon [1, +\infty)\to (0,+\infty)$ be a \emph{differentiable} growth rate.
\begin{definition}
  We define   $\Sigma_{\mu D, A(\cdot)}$ to be the set of all $\tau \in \R$ with the property that the system
\begin{equation}\label{LDE2}
x'=\bigg (A(t)-\tau \frac{\mu'(t)}{\mu(t)} \Id \bigg)x, \quad t\ge 1
\end{equation}
does not admit a strong $\mu$-dichotomy. $\Sigma_{\mu D, A(\cdot)}$ is called the \emph{strong $\mu$-dichotomy spectrum} of \eqref{LDE}.
\end{definition}

Given $\tau\in\R$, let us consider 
\[
T_\tau(t,s)=\bigg (\frac{\mu (t)}{\mu(s)}\bigg )^{-\tau}T(t,s).
\]
Then, differentiating with respect to $t$,
\[
\begin{split}
T_\tau(t,s)'
& =-\tau\left(\frac{\mu(t)}{\mu(s)}\right)^{-\tau-1}\frac{\mu'(t)}{\mu(s)}T(t,s)+\left(\frac{\mu(t)}{\mu(s)}\right)^{-\tau}T(t,s)'\\
& =\left(A(t)-\tau\frac{\mu'(t)}{\mu(t)}\Id \right)\left(\frac{\mu(t)}{\mu(s)}\right)^{-\tau}T(t,s)\\
& =\left(A(t)-\tau\frac{\mu'(t)}{\mu(t)}\Id\right)T_\tau(t,s),\\
\end{split}
\]
and $T_\tau(t,s)$ is the evolution family of~\eqref{LDE2}. Moreover, we have that
\[
T_\tau(n+1,n)=\bigg (\frac{\mu(n+1)}{\mu(n)}\bigg)^{-\tau}T(n+1,n), \quad n\in \N.
\]
Combining this observation with Proposition  \ref{prop: equiv pol dich contXdisc}, we see that, if $T(t,s)$ satisfies \eqref{bnd} and $\mu$ satisfies \eqref{eq: grbound-cont} for some $\theta\geq 1$, then the following properties are equivalent:
\begin{itemize}
\item the sequence $\left(\left(\frac{\mu(n+1)}{\mu(n)}\right)^{-\tau}T(n+1,n)\right)_{n\in \N}$ does not admit a strong $\mu$-dichotomy;
\item \eqref{LDE2} does not admit a strong $\mu$-dichotomy.
\end{itemize}

The following result is a direct consequence of the above fact.

\begin{corollary}\label{cor: equal spectrum cont}
Assume that there exist $K, a>0$ such that~\eqref{bnd} holds and that $\mu$ satisfies \eqref{eq: grbound-cont} for some $\theta\geq 1$. Then
\[
\Sigma_{\mu D, A(\cdot)}=\Sigma_{\mu D, \mathbb A},
\]
where the sequence $\mathbb A=(A_n)_{n\in \N}$ is given by~\eqref{AN}. In particular, $\Sigma_{\mu D, A(\cdot)}$ is given by~\eqref{EDB-mu} with $1\leq r\le d$ and $a_1\leq b_1<a_2\leq \ldots < a_r\leq b_r$.
\end{corollary}

\subsection{$C^1$ linearization under $\mu$-dichotomy} The following is a version of Theorem \ref{theo: main linearization discrete} in the continuous time setting.
\begin{theorem}\label{theo: linarization continuous time pol}
Let $\mu$ be a differentiable growth rate satisfying \eqref{eq: grbound-cont} for some $\theta\geq 1$ and suppose that \eqref{LDE} admits a strong $\mu$-dichotomy. Assume that $\Sigma_{\mu D, A(\cdot)}$ is given by~\eqref{EDB-mu}, where $a_i$ and $b_i$ satisfy~\eqref{sp1} and~\eqref{sp2}.
 Furthermore, suppose that $f\colon [1, \infty)\times \R^d \to \R^d$ is a $C^1$ map satisfying the following conditions:
\begin{itemize}
\item for $t\ge 1$,
\begin{equation}\label{1950}
f(t,0)=0 \quad \text{and} \quad D_x f(t,0)=0;
\end{equation}
\item there exists $\eta >0$ such that 
\begin{equation}\label{1951}
\| D_x f(t,x)\| \le \eta \frac{\mu'(t)}{\mu(t)}, \quad \text{for $x\in \R^d$};
\end{equation}
\item there exists $L>0$ such that 
\begin{equation}\label{1952}
\|D_x f(t,x)-D_x f(t, y)\| \le L\frac{\mu'(t)}{\mu(t)} \|x-y\|, \quad \text{for $x, y\in \R^d$}.
\end{equation}
\end{itemize}
Then, provided that $\eta$ is sufficiently small,  there exist $C^1$ maps $H, G\colon [1, \infty) \times \R^d \to \R^d$ with the following properties:
\begin{itemize}
\item[i)] if $t\mapsto x(t)$ is a solution of~\eqref{NDE}, then $t\mapsto H(t, x(t))$ is a solution of~\eqref{LDE};
\item[ii)] if $t\mapsto x(t)$ is a solution of~\eqref{LDE}, then $t\mapsto G(t, x(t))$ is a solution of~\eqref{NDE};
\item[iii)] for $t\ge 1$ and $x\in \R^d$, 
\[
H(t, G(t,x))=x \quad \text{and} \quad G(t, H(t,x))=x;
\]
\item[iv)] there exist $R, \zeta>0$ such that 
\[
\|D_xH(t,x)\|\leq R \; \text{ and }\; \|D_xG(t,x)\| \leq R , 
\]
for every $t\geq 1$ and $x\in \R^d$ that satisfies $\|x\|\leq \zeta$.
\end{itemize}
\end{theorem}

\begin{proof} The general idea of the proof consists of using the discretization of \eqref{LDE} given by \eqref{AN} to obtain Theorem \ref{theo: linarization continuous time pol} as a consequence of Theorem \ref{theo: main linearization discrete}.

Let $(A_n)_{n\in \N}$ be the sequence of operators on $\R^d$ given by~\eqref{AN}. Combining our assumptions with Proposition \ref{prop: equiv pol dich contXdisc} and Corollary \ref{cor: equal spectrum cont}, we easily see that $(A_n)_{n\in \N}$ satisfies the hypotheses of Theorem \ref{theo: main linearization discrete}. Let $t\mapsto \varphi (t, t_0; x_0)$ denote the solution of~\eqref{NDE} satisfying $x(t_0)=x_0$. Hence, by the variation of constants formula, we have that 
\begin{align}\label{def-phi}
\phi(t, n; x)=T(t,n)x+\int_n^t T(t,r)f(r, \phi(r, n;x))\, dr.
\end{align}
For $n\in \N$ and $x\in \R^d$, set 
\begin{equation}\label{eq: def gn cont case}
g_n(x)=\phi(n+1, n;x)-A_n x=\int_n^{n+1}T(n+1, r)f(r, \phi(r,n;x))\, dr.
\end{equation}
We will now verify that this sequence of $C^1$ maps $g_n:\mathbb{R}^d\to \mathbb{R}^d$, $n\in \N$ satisfies all the assumptions of Theorem~\ref{theo: main linearization discrete}.

It follows from~\eqref{1950} that  $\phi(t, t_0;0)=0$, and thus $g_n(0)=0$ for $n\in \N$. Moreover, 
\[
\begin{split}
Dg_n(0)
&=\int_n^{n+1}T(n+1, r)D_xf(r, 0)D_x \phi(r,n;0)\, dr =0,
\end{split}
\]
for $n\in \N$. We conclude that~\eqref{eq: gn 0 and Dgn 0} is true.

 Observe that 
\begin{equation}\label{dphi}
D_x\phi(t, n; x)=T(t,n)+\int_n^t T(t,r)D_xf(r, \phi(r, n;x))D_x\phi(r, n;x)\, dr
\end{equation}
for $t\ge n$ and $x\in \R^d$.  In particular, from~\eqref{bnd}, \eqref{eq: grbound-cont} and~\eqref{1951} we obtain that for every $t\in [n, n+1]$ and $x\in \R^d$,
\[
\begin{split}
\lVert D_x\phi(t, n; x) \rVert&\le K\bigg (\frac{\mu (t)}{\mu (n)} \bigg )^a  + \int_n^t K \bigg (\frac{\mu(t)}{\mu(r)} \bigg )^a \eta \frac{\mu'(r)}{\mu(r)} \lVert D_x\phi(r, n;x)\rVert\, dr
\\
&\le K\theta^a  +\int_n^t K\eta  \theta^a  \frac{\mu'(r)}{\mu(r)} \lVert D_x\phi(r, n;x)\rVert\, dr.
\end{split}
\]
 Hence, from Gronwall's lemma it follows that 
$\lVert D_x\phi(t, n; x) \rVert \le K \theta^a  e^{\int_n^t K\eta  \theta^a  \frac{\mu'(r)}{\mu(r)}\, dr}$. Thus, since (using \eqref{eq: grbound-cont})
\begin{equation}\label{eq: aux int mu' mu}
\int_n^t   \frac{\mu'(r)}{\mu(r)}\, dr\leq  \ln \mu (n+1)-\ln \mu(n)\leq \ln \theta\leq  \theta,
\end{equation}
we get 
\begin{equation}\label{910x}
\lVert D_x\phi(t, n; x) \rVert \le \hat M 
\end{equation}
for every $t\in [n, n+1]$ and $x\in \R^d$, where $\hat M:=K\theta^ae^{K\eta \theta^{a+1} }$ (which, in particular, is independent of $n$).
On the other hand, note that
\begin{equation}\label{Dfm}
Dg_n(x)=\int_n^{n+1}T(n+1, r)D_xf(r, \phi(r, n;x))D_x\phi(r, n;x)\, dr.
\end{equation}
Then, combining \eqref{bnd}, \eqref{eq: grbound-cont}, \eqref{1951}, \eqref{910x} with \eqref{Dfm} we get
\[
\begin{split}
\lVert Dg_n(x) \rVert &\le \int_n^{n+1}K \bigg (\frac{\mu(n+1)}{\mu(r)} \bigg )^a  \eta \frac{\mu'(r)}{\mu(r)} \hat M  \, dr \\
&\le \frac{K\hat M \theta^{a}\eta }{\mu(n)} \int_n^{n+1}\mu'(r)dr =K\hat M \theta^{a}\eta \left( \frac{\mu(n+1)-\mu(n)}{\mu(n)} \right),
\end{split}
\]
for each $n\in \N$ and $x\in \R^d$. We conclude that \eqref{eq: bound deriv g} holds with $c:=K\hat M\theta^{a}\eta >0$, which can be made sufficiently small by taking $\eta$ sufficiently small. 

Next, we observe that~\eqref{dphi} implies that 
\[
\begin{split}
&D_x\phi(t, n;x)-D_x\phi(t, n; y)
\\
&=\int_n^t T(t,r)D_xf(r, \phi(r, n;x))D_x\phi(r, n;x)\, dr \\
&\phantom{=}-\int_n^t T(t,r)D_xf(r, \phi(r, n;y))D_x\phi(r, n;y)\, dr \\
&=\int_n^t T(t,r)D_xf(r, \phi(r, n;x))(D_x\phi(r, n;x)-D_x\phi(r, n;y))\, dr \\
&\phantom{=}+\int_n^t T(t,r)(D_xf(r, \phi(r, n;x))-D_xf(r, \phi(r, n;y)))D_x\phi(r, n;y)\, dr.
\end{split}
\]
Hence, it follows from~\eqref{bnd}, \eqref{eq: grbound-cont}, \eqref{1951}, \eqref{1952}, \eqref{eq: aux int mu' mu} and~\eqref{910x} (together with the
mean-value theorem) that
\[
\begin{split}
&\lVert D_x\phi(t, n;x)-D_x\phi(t, n; y)  \rVert
\\
&\leq \int_n^t K \bigg (\frac{\mu(t)}{\mu(r)}\bigg )^a L\frac{\mu'(r)}{\mu(r)} \lVert \phi(r, n;x)-\phi(r, n;y)\rVert \hat M \, dr\\
&\phantom{=}+\int_n^t K \bigg (\frac{\mu(t)}{\mu(r)}\bigg )^a \eta \frac{\mu'(r)}{\mu(r)} \lVert D_x\phi(r, n;x)-D_x\phi(r, n;y)\rVert \, dr\\
&\le KL \theta^{a+1}\hat M^2\|x-y\|+ \int_n^t K\theta^a \eta \frac{\mu'(r)}{\mu(r)} \lVert D_x\phi(r, n;x)-D_x\phi(r, n;y)\rVert \, dr,
\end{split}
\]
for $t\in [n, n+1]$ and $x, y\in \R^d$.   By Gronwall's inequality again, one can conclude that there exists $\check M\geq 1$ such that
\begin{equation}\label{Df-Df}
\lVert D_x\phi(t, n;x)-D_x\phi(t, n; y)  \rVert  \le  \check M\lVert x-y\rVert,
\end{equation}
for $n\in \N$, $t\in [n, n+1]$ and $x, y\in \R^d$.

Moreover, since
\[
\begin{split}
&Dg_{n}(x)-Dg_{n}(y)
\\
&=\int_n^{n+1}T(n+1, r)D_xf(r, \phi(r, n;x))D_x\phi(r, n;x)\, dr \\
&\phantom{=}-\int_n^{n+1}T(n+1, r)D_xf(r, \phi(r, n;y))D_x\phi(r, n;y)\, dr \\
&=\int_n^{n+1}T(n+1, r)D_xf(r, \phi(r, n;x))(D_x\phi(r, n;x)-D_x\phi(r, n;y))\, dr \\
&\phantom{=}+\int_n^{n\!+\!1}T(n+1, r)(D_xf(r, \phi(r, n;x))
\!-\!D_xf(r, \phi(r, n;y)))D_x\phi(r, n;y)\, dr, \\
\end{split}
\]
we obtain from \eqref{bnd}, \eqref{eq: grbound-cont}, \eqref{1951}, \eqref{1952}, \eqref{910x} and \eqref{Df-Df} (together with the mean-value
theorem) that
\[
\begin{split}
&\lVert Dg_{n}(x)-Dg_{n}(y)  \rVert
\\
&\le \int_n^{n+1}K \bigg (\frac{\mu(n+1)}{\mu(r)} \bigg )^a \eta \frac{\mu'(r)}{\mu(r)} \check M\lVert x-y\rVert\, dr \\
&\phantom{\le}+\int_n^{n+1}K \bigg (\frac{\mu(n+1)}{\mu(r)} \bigg )^a L\frac{\mu'(r) }{\mu(r)} \hat M  \lVert x-y\rVert \hat M  \, dr \\
&\le \frac{K\check M \theta^{a}\eta}{\mu(n)} \|x-y\| \int_{n}^{n+1}\mu'(r)dr +\frac{K L \hat M^2 \theta^{a} }{\mu(n)} \|x-y\| \int_{n}^{n+1}\mu'(r)dr\\
&= \left(K\check M \theta^{a}\eta + K L \hat M^2 \theta^{a}\right) \frac{\mu(n+1)-\mu(n)}{\mu(n)} \|x-y\|\\
\end{split}
\]
for $n\in \N$ and $x, y\in \R^d$. Then, we conclude that \eqref{4-4-3} holds, and, therefore, all the hypotheses of Theorem \ref{theo: main linearization discrete} are satisfied for $\mathbb{A}=(A_n)_{n\in \N}$ and $(g_n)_{n\in \N}$ defined above. 

Thus, provided that $\eta$ is sufficiently small, by Theorem \ref{theo: main linearization discrete} there exists  a sequence $(\psi_n)_{n\in \N}$ of $C^1$-diffeomorphisms $\psi_n \colon \R^d \to \R^d$ satisfying~\eqref{lincond} and~\eqref{hn-1} (for some $T, \rho >0$).
We now set 
\begin{equation}\label{eq: def of H and G}
    H(t,x)=T(t,n) \psi_n (\phi(n, t;x)) \; \text{ and }\;  G(t,x)=\phi(t, n; \psi_n^{-1}(T(n,t)x)),
\end{equation}
for $n\in \N$, $t\in [n, n+1)$ and $x\in \R^d$. In order to verify that $H(t,x)$ and $G(t,x)$ satisfy properties i) and ii) of our statement, we see from \eqref{eq: def gn cont case} that 
\[
\varphi(n+1,n;x)=A_nx+g_n(x)
\]
for any $n\in \N$ and $x\in \R^d$.
Combining this fact with \eqref{lincond} we get that
\[ 
\begin{split}
H(t,\varphi(t,1;x))&=T(t,n) \psi_n (\phi(n, t;\varphi(t,1;x)))\\
&=T(t,n) \psi_n (\phi(n, 1;x)))\\
&=T(t,n) \psi_n (\phi(n,n-1; \varphi(n-1,1;x)))\\
&=T(t,n) \psi_n ((A_{n-1}+g_{n-1})(\varphi(n-1,1;x)))\\
&=T(t,n) (A_{n-1} \circ \psi_{n-1} (\varphi(n-1,1;x)))\\
&=T(t,n-1) \psi_{n-1} (\varphi(n-1,1;x)),\\
\end{split}
\]
for every $t\in[n,n+1)$ and  $n>1$. The formula above also gives us that 
\[
T(t,n) \psi_n (\phi(n, 1;x)))=
T(t,n-1) \psi_{n-1} (\varphi(n-1,1;x)).
\]
Then, proceeding recursively, we can conclude that 
\[
H(t,\varphi(t,1;x))=T(t,1) \psi_{1} (x) 
\]
for $t\geq 1$. In particular, $H(t,\varphi(t,1;x))$ is a solution of \eqref{LDE} which proves i). Similarly, we can prove that $G$ satisfies property ii) of our statement. 

Now, we observe that
\[
\begin{split}
H(t, G(t, x)) &=T(t, n)\psi_n(\phi(n, t; G(t,x))) \\
&=T(t,n)\psi_n (\phi(n, t; \phi(t, n;\psi_n^{-1}(T(n, t)x)) ) \\
&=T(t,n)\psi_n (\psi_n^{-1}(T(n,t)x)) \\
&=T(t,n)T(n, t)x =x,
\end{split}
\]
for each $x\in \R^d$, $t\in [n, n+1)$ and $n\in \N$, we get $H(t, G(t, x))=x$ for every $t\geq 1$ and $x\in \R^d$. Similarly, we can show that $G(t, H(t, x))=x$ for every $t\geq 1$ and $x\in \R^d$. Consequently, conclusion iii) of our statement also holds true. It remains to show that iv) is satisfied. 

Proceeding as we did to obtain~\eqref{910x}, we find that there exists $\hat M>0$ such that
\begin{equation}\label{Dx-Inv}
\|D_x\phi(n,t;x) \| \le  \hat M
\end{equation}
for $n\in \N$, $t\in [n, n+1)$ and $x\in \R^d$. In particular, recalling that $\varphi(n,t;0)=0$ (which is a consequence of \eqref{1950}) and combining \eqref{Dx-Inv} with the mean-value theorem, we get
\begin{equation}\label{bx1}
\|\phi(n,t;x) \| \le \hat M  \|x\|.
\end{equation}
Therefore, if $\|x\|\leq \frac{ \rho}{\hat M}$ we get that $\|\phi(n,t;x) \| \leq \rho$ for every $t\in[n,n+1)$ and $n\in \N$. Now, since for every
$t\ge 1$ there is an $n\in\mathbb{N}$ such that $t\in [n,n+1)$, combining this observation with \eqref{hn-1}, \eqref{bnd}, \eqref{eq: grbound-cont}, \eqref{Dx-Inv} and the definition of $H$, we get that 
\[
\begin{split}
\|D_xH(t,x)\|&\leq\|T(t,n)\| \|D_x\psi_n (\phi(n, t;x))\|\|D_x\phi(n, t;x)\|\\
&\leq K \theta^a T  \hat M,
\end{split}
\]
whenever $\|x\|\leq \frac{\rho}{\hat M }$. Similarly,
\[
\begin{split}
\|D_xG(t,x)\|&\leq K \theta^a T  \hat M,\\
\end{split}
\]
for every $t\geq 1$ whenever $\|x\|\leq \frac{ \rho}{\hat M }$. Therefore, taking $R=K \theta^a T  \hat M$ and $\zeta =\frac{\rho}{\hat M}$, we obtain iv). This concludes the proof of Theorem \ref{theo: linarization continuous time pol}.
\end{proof}

\begin{example} The following is a continuous time version of the Example \ref{example: main theo discrete}. Let $\mu \colon [1,+\infty)\to (0,+\infty)$ be given by $\mu(t)=1+t$ and define $A,P\colon [1,+\infty)\to \R^{2\times 2}$ by
\[
A(t)=\begin{pmatrix}
-\frac{1}{t} & 0\\
0 & \frac{1}{t}
\end{pmatrix} \text{ and } P(t)=\begin{pmatrix}
1 & 0\\
0 & 0
\end{pmatrix}.
\]
Consider the associated dynamical system given by \eqref{LDE}. Then, the evolution family associated to this system is
\[
T(t,s)=\begin{pmatrix}
\frac{s}{t} & 0\\
0 & \frac{t}{s}
\end{pmatrix} \text{ for every } t,s\geq 1.
\]
We can easily see that \eqref{LDE} admits a strong $\mu$-dichotomy with constants $K=a=\lambda=1$ and projections $P(t)$, $t\geq 1$. Moreover, the discretization of this system defined by \eqref{AN} is given precisely by $\mathbb{A}=(A_n)_{n\in \N}$ and $(P_n)_{n\in \N}$ from the Example \ref{example: main theo discrete}. Therefore, by Corollary \ref{cor: equal spectrum cont}, 
\[
\Sigma_{\mu D,A(\cdot)}=\Sigma_{\mu D,\mathbb{A}}=\{-1,1\}
\] 
and $r=2$, $a_1=b_1=-1$ and $a_2=b_2=1$. Consequently, conditions \eqref{sp1} and \eqref{sp2} are satisfied. 

Considering $f:[1,+\infty)\times \R^2\to \R^2$ given by
\[f(t,(x_1,x_2))=\frac{\eta}{t+1}(\xi(x_1),\xi(x_2)),\]
where $\eta>0$ is a constant and $\xi$ is as in Example \ref{example: main theo discrete}, it follows that conditions \eqref{1950}, \eqref{1951}, and \eqref{1952} are satisfied. In particular, Theorem \ref{theo: linarization continuous time pol} may be applied whenever $\eta>0$ is small enough. Moreover, as in the discrete time case, we can easily verify that \eqref{LDE} does not admit a (nonuniform) strong exponential dichotomy and, therefore, previously available results are not applicable to this example.
\end{example}

\subsection{Differentiable and H\"older linearization under $\mu$-dichotomy} The following is a version of Theorem \ref{theo: diff+holder linearization} in the continuous time case. 

\begin{theorem}\label{theo: diff+holder linearization continuous}
Let \eqref{LDE} and $\mu$ be as in Theorem {\rm\ref{theo: linarization continuous time pol}} and assume that $\Sigma_{\mu D, A(\cdot)}$ is given by~\eqref{EDB-mu}, where $a_i$ and $b_i$ are such that \eqref{sp1} and~\eqref{sp3} are satisfied. Assume that $f\colon [1, \infty)\times \R^d \to \R^d$ is a $C^1$ map that satisfies the conditions \eqref{1950}-\eqref{1952}. Then, provided that $\eta$ is sufficiently small, there exist continuous maps $H, G\colon [1, \infty) \times \R^d \to \R^d$ satisfying properties i), ii) and iii) of Theorem {\rm \ref{theo: linarization continuous time pol}}.
Moreover, there exist constants $\tilde R,\tilde \zeta>0$ such that 
\begin{equation}\label{eq: H is diff 0 cont}
H(t,x)=x+o(\|x\|^{1+\varrho})\; \text{ and }\;
G(t,x)=x+ o(\|x\|^{1+ \varrho})
\end{equation}
as $\|x\|\to 0$ and
\begin{align}\label{eq: H and G are holder}
\begin{split}
&\|H(t,x)-H(t,y)\| \leq \tilde R \|x-y\|^{\alpha_1}\; \text{ and }\; \|G(t,x)-G(t,y)\| \leq  \tilde R  \|x-y\|^{\alpha_1}
\end{split}
\end{align}
for $x,y\in\mathbb{R}^d$ satisfying $\|x\|,\|y\|\le \tilde \zeta $, where $\varrho$ and $\alpha_1$ are given in Theorem \ref{theo: diff+holder linearization}.
\end{theorem}
\begin{proof}
We can proceed exactly as in the proof of Theorem \ref{theo: linarization continuous time pol}, but here we apply Theorem \ref{theo: diff+holder linearization} instead of Theorem \ref{theo: main linearization discrete}.

Using the expressions for $H$ and $G$ given in \eqref{eq: def of H and G} and combining \eqref{001}, \eqref{bnd}, \eqref{eq: grbound-cont}, \eqref{910x}, and \eqref{Dx-Inv}, we obtain \eqref{eq: H and G are holder}. We now show how to establish the first estimate in \eqref{eq: H is diff 0 cont}. The second is similar, and we shall refrain from writing it. Using \eqref{000}, \eqref{bnd}, \eqref{eq: grbound-cont} and \eqref{bx1} we get that for all $n\in \N$ and $t \in [n,n+1)$,
\begin{align*}
\lVert H(t,x)-x\rVert
&=\lVert T(t,n)\psi_n(\varphi(n,t;x))-x\rVert\\
&\le
\lVert T(t,n)\psi_n(\varphi(n,t;x)) -T(t,n)\varphi(n,t;x)\rVert\\
& 
\phantom{\le}+\lVert T(t,n)\varphi(n,t;x)-T(t,n)T(n,t)x\rVert\\
&\le
K\theta^a o(\lVert \varphi(n, t;x)\rVert^{1+ \varrho})
+K\theta^a  \lVert \phi(n,t;x)-T(n,t)x\rVert \\
&\le o(\lVert x\rVert^{1+ \varrho})
+K\theta^a \lVert \phi(n,t;x)-T(n,t)x\rVert
\end{align*}
as $\|x\|\to 0$. Moreover, it follows by \eqref{bnd}, \eqref{eq: grbound-cont}, \eqref{1950}, \eqref{1952}, \eqref{def-phi}, \eqref{eq: aux int mu' mu}, and \eqref{bx1} together with the mean-value theorem that for every $n\in \N$ and $t \in [n,n+1)$,
\[
\begin{split}
\lVert \varphi(n,t;x)-T(n,t)x\rVert
&\le \int_n^{n+1}\lVert T(n,s)f(s, \varphi(n,s;x))\rVert\, ds
\\
&\le \int_n^{n+1} K\theta^a   \sup_{\delta \in (0,1)}\|D_xf(s, \delta \varphi(n,s;x))\|\|\varphi(n,s;x)\|\,ds
\\
&\le \int_n^{n+1} K\theta^a L \frac{\mu'(s)}{\mu(s)}
\|\varphi(n,s; x)\|^2 ds
\\
&\le \int_n^{n+1}K\theta^a L \frac{\mu'(s)}{\mu(s)} \hat M ^2 \|x\|^2 ds
\le K\theta^{a+1}L \hat M ^2 \lVert x\rVert^2.
\end{split}
\]
Combining these observations, we get that
\[
\begin{split}
H(t,x)&=x+ o(\lVert x\rVert^{1+ \varrho})
\end{split}
\]
as $\|x\|\to 0$ (recall that $\varrho<1$). This concludes the proof of the theorem.
\end{proof}

\appendix

\section{} \label{sec: appendix a}

This section serves as an addendum to the proof of the main results in \cite{DZZ}. More precisely, we strengthen the properties of the conjugacies presented in the aforementioned work (property \eqref{conj}) and provide an argument that was missing in the original paper.

\begin{theorem}\label{Ap}
Let $\mathbb A=(A_n)_{n\in \Z}$ be a sequence of invertible operators on $\R^d$ that  admits a strong exponential dichotomy.
Suppose that 
\begin{align*}
\Sigma_{ED, \mathbb A}=\bigcup_{i=1}^r [a_i, b_i],
\end{align*}
where $a_i$ and $b_i$ satisfy \eqref{sp1} and \eqref{sp2}. Moreover, let $f_n\colon \R^d \to \R^d$ be a sequence of $C^1$ maps such that \eqref{fzero} and \eqref{fm} hold for $n\in\mathbb{Z}$.
Then, provided that $\eta$ is sufficiently small,  there exists a sequence $h_n\colon \R^d\to \R^d$ of $C^1$-diffeomorphisms such that
\begin{itemize}
\item for $n\in \Z$,
\begin{align}\label{con-hn}
h_{n+1} \circ (A_n+f_n)=A_n\circ h_n;
\end{align}
\item there exist $M, \rho>0$ such that
\begin{equation}\label{conj}
\|D h_n(v)z\| \le M\|z\| \quad \text{and} \quad \|Dh_n^{-1}(v)z\| \le M\|z\|, \quad \forall n\in \Z,
\end{equation}
for all  $z\in \R^d$ and $v\in\mathbb{R}^d$ such that $\|v\|\le \rho$.
\end{itemize}
\end{theorem}

\begin{remark}
This result is essentially established in~\cite[Theorem 2]{DZZ}, except that the conclusion~\eqref{conj} was not explicitly written.  Notice that such $C^1$ conjugacy $h_n$ without satisfying \eqref{conj} always exists by a recursive construction.
\end{remark}

\begin{remark}\label{remark: diff def spect}
We emphasize the difference between the spectral conditions in Theorem \ref{Ap} and \cite[Theorem 2]{DZZ}, which arises from the differing definitions of spectra employed in these works. Specifically, a value $a$ belongs to the strong exponential dichotomy spectrum as defined in \cite{DZZ} if and only if $\ln a$ belongs to the strong exponential dichotomy spectrum as defined in the present paper.
\end{remark}


\begin{proof}[Proof of Theorem \ref{Ap}] Set
\[
Y_\infty:=\bigg \{ \mathbf x=(x_n)_{n\in \Z} \subset \R^d: \|\mathbf x\|_\infty:=\sup_{n\in \Z} \|x_n\| <+\infty \bigg \}.
\]
Then, $(Y_\infty, \| \cdot \|_\infty)$ is a Banach space. Define a bounded linear operator $ \mathbb A ^\ast \colon Y_\infty \to Y_\infty$  by 
\begin{equation}\label{operator}
(\mathbb A ^\ast \mathbf x)_n:=A_{n-1}x_{n-1}, \quad \forall n\in \Z, \quad \forall \mathbf x=(x_n)_{n\in \Z}\in Y_\infty.
\end{equation}
It follows from~\eqref{bg} with $\mu_n=e^n$ for $n\in \Z$ that $\mathbb A^\ast$ is indeed well-defined, bounded and invertible, where the inverse is given by 
\[
((\mathbb A^\ast) ^{-1}\mathbf x)_n=A_n^{-1}x_{n+1}, \quad \mathbf x=(x_n)_{n\in \mathbb Z}\in Y_\infty.
\]
Define the spectrum of $\mathbb A ^\ast $ by
\[
\sigma(\mathbb A ^\ast):=\{\varrho\in \mathbb C: \varrho\Id-\mathbb A ^\ast ~\mbox{is not invertible on $Y_\infty$}\}.
\]
By \cite[Lemma 2]{DZZ}, we see that $\varrho\in \sigma(\mathbb A^\ast)$ (actually considering the complexification of $\mathbb A^\ast$ and $Y_\infty$)
if and only if $\left(\frac{1}{|\varrho|}A_n\right)_{n\in \Z}$ does not admit a strong exponential dichotomy. This fact shows (recall Definition~\ref{def: mu spect} and Remark~\ref{aaa}) that
\[
\Sigma_{ED, \mathbb A}=\{\ln |\lambda |: \  \lambda \in \sigma(\mathbb A^\ast)\}.
\]

Let us consider
\[
|\sigma(\mathbb A^\ast)|=\{ |\lambda|: \ \lambda \in \sigma(\mathbb A^\ast)\}.
\]
Then, by the discussion above we have that
\[
|\sigma(\mathbb A^\ast)|=\bigcup_{i=1}^r [e^{a_i}, e^{b_i}].
\]
Next, using \eqref{fzero} and \eqref{fm} together with the mean-value theorem we get that 
\[
\|f_n(x)\|=\|f_n(x)-f_n(0)\|\le \eta \|x\|,
\]
for every $x\in \mathbb R^d$ and $n\in \mathbb Z$. This implies that the map $\mathbb F\colon Y_\infty \to Y_\infty$ given by
\begin{equation}\label{eq: def F lifting}
(\mathbb F(\mathbf x))_n:=F_{n-1}(x_{n-1}):= A_{n-1}x_{n-1}+f_{n-1}(x_{n-1})
\end{equation}
for $n\in \Z$ and $\mathbf x=(x_n)_{n\in \Z}\in Y_\infty$ is well-defined. According to the proof of \cite[Theorem 2]{DZZ}, we get the following conclusions:
\begin{enumerate}[label={(C\arabic*)},ref=C\arabic*]
\item \label{Df 0 is A} $\mathbb F\in C^{1,1}$ and $D \mathbb F(\mathbf 0)=\mathbb A^\ast$;
\item \label{Df x close A} there exists $C>0$ such that $\| D\mathbb F(\mathbf x)-\mathbb A^\ast\| \le C\eta$ for $\mathbf x\in Y_\infty$.
\end{enumerate}
Hence, provided that $\eta$ is sufficiently small, it follows from the $C^1$ linearization result given in~\cite[Appendix]{DZZ} that there exists a $C^1$-diffeomorphism  $\Phi \colon Y_\infty \to Y_\infty$ such that
\begin{equation}\label{P}
\Phi \circ \mathbb F=\mathbb A^\ast\circ \Phi.
\end{equation}
For $v\in \R^d$ and $m\in \Z$, we set
\begin{equation}\label{eq: def hn}
h_m(v):=(\Phi(\mathbf v^m))_m,\qquad \forall v\in \mathbb{R}^d,~\forall m\in\mathbb{Z},
\end{equation}
where $\mathbf v^m:=(v_n^m)_{n\in \Z}$ is given by $v_m^m=v$ and $v_n^m=0$ for $n\neq m$. Then, it is verified in Proposition \ref{prop: hn is homeo} below that $h_m\colon \R^d \to \R^d$ is a bijection while in the proof of~\cite[Theorem 2]{DZZ} it is shown that $h_m$ and $h_m ^{-1}$ are $C^1$ maps satisfying \eqref{con-hn}.
Moreover, we know that
\begin{equation}\label{DhDPhi}
  D h_m(v) z=(D\Phi(\mathbf v^m)\mathbf z^m)_m \quad \text{and} \quad Dh_m^{-1}(v)z=(D\Phi^{-1}(\mathbf v^m)\mathbf z^m)_m,
\end{equation}
for every $v, z\in \R^d$ and $m\in \Z$, 
where $\mathbf z^m$ is defined as $\mathbf v^m$, replacing $v$ by $z$.

In what follows, we prove \eqref{conj}.
It follows easily from~\eqref{P} that $\Phi(\mathbf 0)=\mathbf 0$.
Since $\Phi$ is $C^1$, there exists constants $M, \rho>0$ such that
\begin{equation}\label{P1}
\|D\Phi(\mathbf x)\| \le M \quad  \text{and} \quad  \|D\Phi^{-1}(\mathbf x)\|\le M, 
\end{equation}
for $\mathbf x\in Y_\infty$ satisfying $\|\mathbf x\|_\infty \le \rho$.
Then, we obtain from \eqref{DhDPhi} and~\eqref{P1} that
\[
\|D h_m(v)z\|\le \|D\Phi(\mathbf v^m)\mathbf z^m\|_\infty
\le \|D\Phi(\mathbf v^m)\|\,\|\mathbf z^m\|_\infty
=M\|z\|,
\]
for $\|v\| \le \rho$ (which implies that $\|\mathbf v^m\|_\infty\le \rho$). Similarly,
\[
\|D h_m^{-1}(v)z\|\le M\|z\|  \text{ for $\|v\|\le \rho$.}
\]
This proves (\ref{conj}), and the proof of the theorem is completed.
\end{proof}

The rest of this appendix is dedicated to showing the following result. We observe that this argument was missing from the original proof of \cite[Theorem 2]{DZZ}.
\begin{proposition}\label{prop: hn is homeo}
Under the assumptions of Theorem \ref{Ap}, for each $m\in \Z$, the map $h_m\colon \R^d \to \R^d$ given by \eqref{eq: def hn} is a bijection.
\end{proposition}

The main step in the proof of Proposition \ref{prop: hn is homeo} consists of showing that $\Phi$ as in \eqref{P} has the property that its $n$-th coordinate depends only on the $n$-th coordinate of the argument. In order to prove this, we need several auxiliary observations. We retain all the notation introduced in the proof of Theorem \ref{Ap}.

Observe initially that condition \eqref{pd1} with $\mu_n=e^n$ for $n\in \Z$ can be translated into
\begin{equation}\label{dic1}
\|\mathcal A(m, n)P_n\|\le Ke^{-\lambda (m-n)} \text{ for $m\ge n$}
\end{equation}
and
\begin{equation}\label{dic2}
\|\mathcal A(m, n)(\Id-P_n)\|\le K e^{-\lambda (n-m)} \text{ for $m\le n$.}
\end{equation}
This implies that $\mathbb A^\ast $ is hyperbolic. That is, if we denote by $\Pi^s\colon Y_\infty \to Y_\infty$ the projection given by 
\[
(\Pi^s \mathbf x)_n=P_n x_n, \quad \mathbf x=(x_n)_{n\in \mathbb Z}\in Y_\infty,
\]
then
\[
\|(\mathbb A^\ast)^n \Pi^s\|\le Ke^{-\lambda n} \; \text{ for } n\ge 0,
\]
and 
\[
\|(\mathbb A^\ast)^{n}\Pi^u\|\le Ke^{\lambda n} \; \text{ for } n\le 0,
\]
where $\Pi^u:=\Id-\Pi^s$. Note that these estimates follow readily from~\eqref{dic1} and~\eqref{dic2}.


\subsection{Foliations}\label{sec: foliations}
Due to the hyperbolicity of $\mathbb A ^{\ast}$ and~\eqref{Df x close A}, $\mathbb F$ has stable and unstable global foliations (provided that $\eta$ is sufficiently small). More precisely, let $Y_\infty^s=\Pi^s Y_\infty$ and $Y_\infty^u=\Pi^u Y_\infty$. Then the stable leaf at $\mathbf x\in Y_\infty$ is given by 
\[
\mathcal M_s(\mathbf x)=\left \{\mathbf y\in Y_\infty: \ \sup_{n\in \mathbb N}(e^{\varrho n}\|\mathbb F^n(\mathbf x)-\mathbb F^n(\mathbf y)\|_\infty)<+\infty \right  \},
\]
where $\varrho \in (0, \lambda)$ is sufficiently small (independently of $\mathbf x$). By~\cite[Appendix]{DZZ},  there exists $p\colon Y_\infty \times Y_\infty^s \to Y_\infty^u$ of class $C^1$ such that 
\[
\mathcal M_s(\mathbf x)=\{\mathbf v+p(\mathbf x, \mathbf v): \ \mathbf v\in Y_\infty^s\}.
\]
Furthermore, $p(\mathbf x, \Pi^s \mathbf x)=\Pi^u \mathbf x$.
We claim that the value $p(\mathbf x, \mathbf v)_n$ depends only on the $n$-th coordinates of $\mathbf x$ and $\mathbf v$.
\begin{lemma}\label{lem: p x k depends only vk}
Take $\mathbf x\in Y_\infty$ and 
let $\mathbf v^i\in Y_\infty^s$, $i=1, 2$ be such that $v_k^1=v_k^2$ for some $k\in \mathbb Z$. Then
\[
p(\mathbf x, \mathbf v^1)_k=p(\mathbf x, \mathbf v^2)_k.
\]
\end{lemma}

\begin{proof}
We define $\mathbf y=(y_n)_{n\in \mathbb Z}\in Y_\infty$ by
\[
\begin{split}
y_n &=\begin{cases}
v_n^2+p(\mathbf x, \mathbf v^2)_n & n\neq k; \\
v_k^1+p(\mathbf x, \mathbf v^1)_k & n=k
\end{cases} \\
&=\begin{cases}
v_n^2+p(\mathbf x, \mathbf v^2)_n & n\neq k; \\
v_k^2+p(\mathbf x, \mathbf v^1)_k & n=k.
\end{cases}
\end{split}
\]
Observe that for $n\in \mathbb N$ and $m\in \mathbb Z$ we have 
\[
\begin{split}
&(\mathbb F^n(\mathbf x)-\mathbb F^n(\mathbf y))_m \\
&=
\begin{cases}
\mathcal F(m, m-n)(x_{m-n})-\mathcal F(m, m-n)(v_{m-n}^2+p(\mathbf x, \mathbf v^2)_{m-n}) & m\neq n+k; \\
\mathcal F(m, m-n)(x_{m-n})-\mathcal F(m, m-n)(v_{m-n}^1+p(\mathbf x, \mathbf v^1)_{m-n}) &m=n+k, 
\end{cases}
\end{split}
\]
where 
\[
\mathcal F(k, l):=\begin{cases}
F_{k-1} \circ \ldots \circ F_l & k>l; \\
\Id & k=l,
\end{cases}
\]
and $F_k$ is given in \eqref{eq: def F lifting}.
Consequently,
\[
\|\mathbb F^n(\mathbf x)-\mathbb F^n(\mathbf y)\|_\infty\le \max \{\|\mathbb F^n(\mathbf x)-\mathbb F^n(\mathbf v^2+p(\mathbf x, \mathbf v^2))\|_\infty,  \|\mathbb F^n(\mathbf x)-\mathbb F^n(\mathbf v^1+p(\mathbf x, \mathbf v^1))\|_\infty \}.
\]
Thus, since $\mathbf v^i+p(\mathbf x, \mathbf v^i)\in  \mathcal M_s(\mathbf x)$ for $i=1,2$, it easily follows that $\mathbf y\in \mathcal M_s(\mathbf x)$. Therefore, there is $\mathbf v\in Y_\infty^s$ such that
\[
\mathbf y=\mathbf v+p(\mathbf x, \mathbf v).
\]
By projecting onto $Y_\infty^s$, we see that $\mathbf v=\mathbf v^2$. Hence, 
\[
v_k^1+p(\mathbf x, \mathbf v^1)_k=y_k=v_k^2+p(\mathbf x, \mathbf v^2)_k=v_k^1+p(\mathbf x, \mathbf v^2)_k,
\]
yielding the desired claim.
\end{proof}

\begin{lemma}\label{lem p x k depends only on xk}
Let $\mathbf x^i\in Y_\infty$, $i=1, 2$ and $\mathbf v\in Y_\infty^s$ be such that $x_k^1=x_k^2$ for some $k\in \mathbb Z$. Then
\[
p(\mathbf x^1, \mathbf v)_k=p(\mathbf x^2, \mathbf v)_k.
\]
\end{lemma}

\begin{proof}
We define $\mathbf y=(y_n)_{n\in \mathbb Z}\in Y_\infty$ by 
\[
y_n=\begin{cases}
v_n +p(\mathbf x^2, \mathbf v)_n &n\neq k; \\
v_k+p(\mathbf x^1, \mathbf v)_k & n=k.
\end{cases}
\]
Then, using the notation introduced in the proof of Lemma \ref{lem: p x k depends only vk}, for $n \in \mathbb N$ and $m\in \mathbb Z$, we have 
\[
\begin{split}
&(\mathbb F^n (\mathbf x^2)-\mathbb F^n(\mathbf y))_m \\
&=\begin{cases}
\mathcal F(m, m-n)(x_{m-n}^2)-\mathcal F(m, m-n)(v_{m-n}+p(\mathbf x^2, \mathbf v)_{m-n})& m\neq n+k; \\
\mathcal F(m, m-n)(x_{m-n}^1)-\mathcal F(m, m-n)(v_{m-n}+p(\mathbf x^1, \mathbf v)_{m-n}) & m=n+k.
\end{cases}
\end{split}
\]
Hence, 
\[
\|\mathbb F^n (\mathbf x^2)-\mathbb F^n(\mathbf y)\|_\infty \le \max \{\|\mathbb F^n(\mathbf x^2)-\mathbb F^n(\mathbf v+p(\mathbf x^2, \mathbf v))\|_\infty, \|\mathbb F^n(\mathbf x^1)-\mathbb F^n(\mathbf v+p(\mathbf x^1, \mathbf v))\|_\infty\},
\]
and thus $\mathbf y\in \mathcal M_s(\mathbf x^2)$. Hence, there exists $\tilde{\mathbf v}\in Y_\infty^s$ such that
\[
\mathbf y=\tilde{\mathbf v}+p(\mathbf x^2, \tilde{\mathbf v}).
\]
In particular, $\Pi^s\mathbf y=\tilde{\mathbf v}$. On the other hand, from the definition of $\mathbf y$ we have $\Pi^s\mathbf y=\mathbf v$. Thus, $\tilde{\mathbf v}=\mathbf v$. Therefore, the equality above, together with the definition of $\mathbf y$, implies that
\[
v_k+p(\mathbf x^1, \mathbf v)_k =y_k=v_k+p(\mathbf x^2, \mathbf v)_k,
\]
yielding the desired claim.
\end{proof}

\begin{corollary}
Let $\mathbf x^i\in Y_\infty$ and $\mathbf v^i\in Y_\infty^s$, $i=1, 2$ be such that $x_k^1=x_k^2$ and $v_k^1=v_k^2$ for some $k\in \mathbb Z$. Then
\[
p(\mathbf x^1, \mathbf v^1)_k=p(\mathbf x^2, \mathbf v^1)_k.
\]
\end{corollary}

\begin{proof}
Using Lemmas \ref{lem: p x k depends only vk} and \ref{lem p x k depends only on xk}, we have 
\[
p(\mathbf x^1, \mathbf v^1)_k=p(\mathbf x^1, \mathbf v^2)_k=p(\mathbf x^2, \mathbf v^2)_k
\]
which proves the corollary.
\end{proof}
It follows from the above corollary that $p$ has the form
\[
p(\mathbf x, \mathbf v)=(p_n(x_n, v_n))_{n\in \mathbb Z},
\]
for some sequence of maps $p_n\colon \mathbb R^d \times \Ima P_n \to \Ker P_n$. The above also show that the stable manifold $\mathcal M_s(\mathbf 0)$ (where $\mathbf 0=(0)_{n\in \mathbb Z}$) is the graph of the function $g^s \colon Y_\infty^s \to Y_\infty^u$ given by $g^s(\mathbf v):=p(\mathbf 0, \mathbf v)$. Observe that $g^s(\mathbf v)_n$ depends only on $v_n$.

The same discussion applies to unstable foliation. Thus, for $\mathbf x\in Y_\infty$, the unstable leaf \[\mathcal M_u(\mathbf x)=\left \{\mathbf y\in Y_\infty: \  \sup_{n\in \mathbb N}(e^{\varrho n}\|\mathbb F^{-n}(\mathbf x)-\mathbb F^{-n} (\mathbf y)\|_\infty) <+\infty \right \}\] can be represented as 
\[
\mathcal M_u(\mathbf x)=\{\mathbf v+q(\mathbf x, \mathbf v): \ \mathbf v\in Y_\infty^u\},
\]
where $q\colon Y_\infty \times Y_\infty^u \to Y_\infty^s$ is of class $C^1$ such that $q(\mathbf x, \mathbf v)_n$ depends only on $x_n$ and $v_n$. The unstable manifold $\mathcal M_u(\mathbf 0)$ is the graph of the function $g^u:=q(\mathbf 0, \cdot) \colon Y_\infty^u \to Y_\infty^s$ such that $g^u(\mathbf v)_n$ depends only on $v_n$. 

\subsection{Splitting}\label{splitting}
Our next goal is to show that the homeomorphisms $\eta, \sigma \colon Y_\infty \to Y_\infty$ given~\cite[Lemma 3.1]{BT} also have the property that $\eta(\mathbf x)_n$ and $\sigma(\mathbf x)_n$ depend only on $x_n$.  We note that since $p$ and $q$ are of class $C^1$, $\eta$ and $\sigma$ are also of class $C^1$ (see~\cite[Lemma 3.1]{BT}).

To this end,  we note that the construction of $\eta$ in~\cite[Lemma 3.1]{BT} gives that 
$\eta(\mathbf x)$ is a fixed point for the contraction $G_\mathbf x\colon Y_\infty \to Y_\infty$ defined by 
\[
G_\mathbf x(\mathbf y):=p(\mathbf x, g^u(\Pi^u \mathbf y))+q(\mathbf x, g^s(\Pi^s\mathbf y)), \quad \mathbf y\in Y_\infty.
\]
Since $(\Pi^s \mathbf y)_n$ and $(\Pi^u \mathbf y)_n$ depend only on $y_n$, it follows from the discussion in the previous section that $G_\mathbf x(\mathbf y)_n$ depends only on $x_n$ and $y_n$.

\begin{lemma}\label{lem eta n depend on xn}
For $\mathbf x, \tilde{\mathbf x}\in Y_\infty$ such that $x_n=\tilde x_n$ for some $n\in \mathbb Z$, we have $\eta(\mathbf x)_n=\eta(\tilde{\mathbf x})_n$.
\end{lemma}

\begin{proof}
Since $G_{\cdot}$ is a contraction, we have 
\[
\eta(\mathbf x)=\lim_{m\to \infty}G_{\mathbf x}^m(\mathbf 0) \; \text{ and }\;  \eta(\tilde{\mathbf x})=\lim_{m\to \infty}G_{\tilde{\mathbf x}}^m(\mathbf 0).
\]
We now show that for each $m\in \mathbb N$, 
\begin{equation}\label{nn}
G_{\mathbf x}^m(\mathbf 0)_n=G_{\tilde{\mathbf x}}^m(\mathbf 0)_n.
\end{equation}
For $m=1$ the desired conclusion follows from the discussion that precedes the statement of the lemma. Assume that it holds for some $m\in \mathbb N$. Then 
\[
\begin{split}
G_{\mathbf x}^{m+1}(\mathbf 0)_n= (G_{\mathbf x} (G_{\mathbf x}^m(\mathbf 0))_n &=(G_{\mathbf x} (G_{\tilde{\mathbf x}}^m(\mathbf 0))_n  \quad \text{(since $G_{\mathbf x}^m(\mathbf 0)_n=G_{\tilde{\mathbf x}}^m(\mathbf 0)_n$)}\\
&=(G_{\tilde{\mathbf x}} (G_{\tilde{\mathbf x}}^m(\mathbf 0))_n \quad \text{(since $x_n=\tilde x_n$)}\\
&=G_{\tilde{\mathbf x}}^{m+1}(\mathbf 0)_n,
\end{split}
\]
yielding~\eqref{nn} for $m+1$. Thus, the claim follows by induction.

Consequently,
\[
\eta(\mathbf x)_n=\lim_{m\to \infty}G_\mathbf x^m(\mathbf 0)_n=\lim_{m\to \infty}G_{\tilde {\mathbf x}}^m(\mathbf 0)_n=\eta(\tilde{\mathbf x})_n.
\]
\end{proof}

For $n\in \mathbb Z$, we define $\eta_n \colon \mathbb R^d \to \mathbb R^d$ by 
\[
\eta_n(v):=\eta (\mathbf v)_n \quad \text{for $v\in \mathbb R^d$,}
\]
where $\mathbf v=(v_m)_{m\in \mathbb Z}\in Y_\infty$ is any such that $v_n=v$. Note that $\eta_n$ is well-defined due to Lemma \ref{lem eta n depend on xn}. Since $\eta$ is a homeomorphism, we find that $\eta_n$ is a homeomorphism for each $n$. Indeed, take $v, w\in \mathbb R^d$ such that $\eta_n(v)=\eta_n(w)$. Let $\mathbf v=(v_m)_{m\in \mathbb Z}, \mathbf w=(w_m)_{m\in \mathbb Z}\in Y_\infty$ be such that $v_n=v$, $w_n=w$, and $v_m=w_m=0$ for $m\neq n$.
Observe that $\eta(\mathbf v)=\eta(\mathbf w)$ as $\eta(\mathbf v)_m=\eta(\mathbf w)_m$ for each $m\in \mathbb Z$ (recall Lemma \ref{lem eta n depend on xn}). Hence, $\mathbf v=\mathbf w$ and $v=w$ and we conclude that $\eta_n$ is injective. Now, take an arbitrary $w\in \mathbb R^d$ and choose $\mathbf w=(w_m)_{m\in \mathbb Z}\in Y_\infty$ such that $w_n=w$. Since $\eta$ is onto, there is $\mathbf v=(v_m)_{m\in \mathbb Z}\in Y_\infty$ such that $\eta(\mathbf v)=\mathbf w$. Then
\[
w=w_n=\eta(\mathbf v)_n=\eta_n(v_n),
\]
showing that $\eta_n$ is onto.

Note that
\[
\eta(\mathbf  v)=(\eta_n(v_n))_{n\in \mathbb Z} \quad \text{for $\mathbf v=(v_n)_{n\in \mathbb Z}\in Y_\infty$,}
\]
and 
\[
\eta^{-1}(\mathbf  v)=(\eta_n^{-1}(v_n))_{n\in \mathbb Z} \quad \text{for $\mathbf v=(v_n)_{n\in \mathbb Z}\in Y_\infty$.}
\]
Since $\sigma\colon Y_\infty \to Y_\infty$ given by~\cite[Lemma 3.1]{BT} is the inverse of $\eta$, we immediately get the following result.
\begin{lemma}
For $\mathbf x, \tilde{\mathbf x}\in Y_\infty$ such that $x_n=\tilde x_n$ for some $n\in \Z$, we have $\sigma(\mathbf x)_n=\sigma(\tilde{\mathbf x})_n$.
\end{lemma}

We define $\mathbb F_{-}\colon Y_\infty^s \to Y_\infty^s$ and $\mathbb F_+\colon Y_\infty^u \to Y_\infty^u$ by 
\begin{equation*}\label{F1}
\mathbb F_{-}:=\Pi^s \circ \mathbb F\circ (\Id_{Y_\infty^s}+g^s)
\end{equation*}
and 
\begin{equation*}\label{F2}
\mathbb F_+:=\Pi^u \circ \mathbb F\circ (\Id_{Y_\infty^u}+g^u).
\end{equation*}
Then, \cite[Lemma 3.2]{BT} gives that
\begin{equation}\label{lin1}
\eta\circ \mathbb F=(\mathbb F_{-}\circ \Pi^s+\mathbb F_+\circ \Pi^u)\circ \eta
\end{equation}
and
\[
\mathbb F\circ \sigma=\sigma \circ (\mathbb F_{-}\circ \Pi^s+\mathbb F_+\circ \Pi^u).
\]
We note that $\mathbb F_{\pm}(\mathbf x)_n$ depends only on $x_{n-1}$ and $D\mathbb F_{-}(\mathbf 0)=\mathbb A_-^\ast:=\mathbb A^\ast\rvert_{Y_\infty^s}$ and $D\mathbb F_+(\mathbf 0)=\mathbb A_+^\ast:=\mathbb A^\ast\rvert_{Y_\infty^u}$.

\subsection{Linearization of a contractive system}\label{sec: linear contraction}
We now turn our attention to the linearization of contractive systems. We will follow the idea developed in \cite[Sections 5 and 7]{ZZJ}. The case of expanding systems can be treated in an analogous manner.

Let us assume that the linear part $(A_n)_{n\in \mathbb Z}$ admits an exponential contraction (that is, an exponential dichotomy with $P_n=\Id$ for every $n\in\Z$) and that the corresponding Sacker-Sell spectrum is of the form (recall Remark \ref{remark: diff def spect})
\[
a_1\le b_1<a_2\le b_2\le \ldots \le a_k\le b_k <0.
\]

\subsubsection{Local linearization}
We can assume that $A_n$, $n\in \mathbb Z$ are of the block-diagonal form $A_n=\diag (A_n^1, \ldots, A_n^k)$, $n\in \mathbb  Z$, where the Sacker-Sell spectrum of $(A_n^i)_{n\in \mathbb Z}$ is given by $[a_i, b_i]$ for each $i\in \{1, \ldots, k\}$ (apply~\cite[Theorem 3.11]{CZSX} to the standard dichotomy spectrum).  Assume that the dimension of the block $A_n^i$ is $d_i$ for $i\in \{1, \ldots, k\}$. Notice that the decomposition $\mathbb R^d=\bigoplus_{i=1}^k \mathbb R^{d_i}$ induces the decomposition
\[
Y_\infty=\bigoplus_{i=1}^k Y_\infty^i,
\]
where $Y_\infty^i$ consists of all $\mathbf x=(x_n)_{n\in \mathbb Z}\in Y_\infty$ with $x_n\in \mathbb R^{d_i}$ for each $n$.

Now, the general strategy in the proof of the main result in \cite{ZZJ} consists in finding a sequence of $C^1$ diffeomorphisms $(\Phi_i)_{i=1}^k$ such that $\Phi_i$ is a conjugacy between $\mathbb F_i$ and $\mathbb F_{i-1}$ where $\mathbb F_k= \mathbb F$ and $\mathbb F_0=\mathbb A^\ast$ and, for $i=1,\ldots,k-1 $, each $\mathbb F_i$ has the form $\mathbb F_i(\mathbf x)=(F_n^i(x_n))_{n\in \Z}$ for $\mathbf x=(x_n)_{n\in \Z}\in Y_\infty$ with 
\[F_n^i(x)=\begin{pmatrix}
    A^1_nx^1&+&f^1_{in}(x^1,x^2,\ldots,x^k)\\
    A^2_nx^2&+&f^2_{in}(x^1,x^2,\ldots,x^k)\\
    &\vdots & \\
    A^i_nx^i&+&f^i_{in}(x^1,x^2,\ldots,x^k)\\
    A^{i+1}_nx^{i+1}&&\\
    &\vdots&\\
    A^k_nx^k&&     
\end{pmatrix}\]
where $x=(x^1,x^1,\ldots,x^k)\in \oplus_{i=1}^k \mathbb R^{d_i}$ and $f^j_{in}\colon \R^d\to \R^{d_j}$. Then, considering the composition $\Phi:=\Phi_1\circ\ldots\circ\Phi_k$, we get a conjugacy between $\mathbb F$ and $\mathbb A$. The objective in the sequel is to observe that $\Phi$ has the property that its $n$-th coordinate depends only on the $n$-th coordinate of the argument.

As explained in the proof of~\cite[Lemma 10]{ZZJ}, the induction step starts from $\mathbb F_k:=\mathbb F$, i.e.,
\[
\mathbb F_k(\mathbf x)=(F_{n-1}(x_{n-1}))_{n\in \mathbb Z}, \quad \mathbf x=(x_n)_{n\in \mathbb Z}\in Y_\infty.
\]
By writing a point in $x\in \mathbb R^d$ as a pair of $(u, v)$ with $u\in \bigoplus_{i=1}^{k-1}\mathbb R^{d_i}$ and $v\in \mathbb R^{d_k}$, we can write each $F_n$ as 
\[
F_n(u, v)=(\diag (A_n^1, \ldots, A_n^{k-1})u+\tilde{f}_n^a(u, v), A_n^k v+\tilde{f}_n^b(u, v)),
\]
where $\tilde{f}_n^*=\pi^* \circ f_n$, $*=a, b$ and $\pi^a\colon \mathbb R^d\to \bigoplus_{i=1}^{k-1}\mathbb R^{d_i}$ and $\pi^b\colon \mathbb R^d \to \mathbb R^{d_k}$ are projections associated with the splitting $\mathbb R^d=(\bigoplus_{i=1}^{k-1}\mathbb R^{d_i})\oplus \mathbb R^{d_k}$. 

Consequently, each point $\mathbf x\in Y_\infty$ can be written as $(\mathbf u, \mathbf v)$, where $\mathbf u\in \bigoplus_{i=1}^{k-1}Y_\infty^i$ and $\mathbf v\in Y_\infty^k$ and
\[
\mathbb F_k(\mathbf u, \mathbf v)=(F_{n-1}(u_{n-1}, v_{n-1}))_{n\in \mathbb Z}, \quad \mathbf u\in \bigoplus_{i=1}^{k-1}Y_\infty^i, \ \mathbf v\in Y_\infty^k.
\]
Recall the map $\Psi_k$ defined in~\cite[Section 5]{ZZJ} which is given by
\[
\Psi_k:=\lim_{m\to \infty}(\mathbb A^\ast_k)^{-m}\Pi^b\mathbb F_k^m,
\]
where $\Pi^b \colon Y_\infty \to Y_\infty^k$ is the projection given by $Y_\infty \ni (x_n)_{n\in \mathbb Z}\mapsto (\pi^b x_n)_{n\in \mathbb Z}$, and $\mathbb A_k^\ast:=\mathbb A^\ast \rvert_{Y_\infty^k}$.
Since $\mathbb F_k^m(\mathbf u, \mathbf v)_n$ depends only on $(u_{n-m}, v_{n-m})$ and $((\mathbb A^\ast_k)^{-m}\mathbf x)_n$ depends only on $x_{n+m}$, we conclude that $\Psi_k(\mathbf u, \mathbf v)_n$ depends only on $(u_n, v_n)$. The same applies then to the map $\Phi_k(\mathbf u, \mathbf v):=(\mathbf u, \Psi_k(\mathbf u, \mathbf v))$ introduced in the proof of~\cite[Lemma 10]{ZZJ}.

Let $\mathbb F_{k-1}$ be given by~\cite[Eq. (5.14)]{ZZJ} (for $\ell=k$):
\[
\mathbb F_{k-1}:=\Phi_k \circ \mathbb F_{k}\circ \Phi_k^{-1},
\]
defined locally around $\mathbf 0\in Y_\infty$. Due to the property of $\Phi_k$ mentioned above (which also holds for its local inverse), $\mathbb F_{k-1}(\mathbf x)_n$ depends only on $x_{n-1}$.
Writing a point $\mathbf x$ as a triple $(\mathbf u, \mathbf v, \mathbf w)$ with $\mathbf u\in \bigoplus_{i=1}^{k-2}Y_\infty^i$, $\mathbf v\in Y_\infty^{k-1}$ and $\mathbf w\in Y_\infty^k$, we have (see the proofs of~\cite[Lemma 5 and Lemma 10]{ZZJ}) that 
\[
\begin{split}
\mathbb F_{k-1}(\mathbf u, \mathbf v, \mathbf w)
&=(\diag (A_{n-1}^1, \ldots, A_{n-1}^{k-2})u_{n-1}+\bar f_{n-1}^a(u_{n-1}, v_{n-1}, w_{n-1}), A_{n-1}^{k-1}v_{n-1}
\\
&\quad +\bar f_{n-1}^b(u_{n-1}, v_{n-1}, w_{n-1}), A_{n-1}^k w_{n-1})_{n\in \mathbb Z},
\end{split}
\]
where $\bar f_n^a \colon \mathbb R^d \to \bigoplus_{i=1}^{k-2}\mathbb R^{d_i}$ and $\bar f_n^b\colon \mathbb R^d \to \mathbb R^{d_{k-1}}$, $n\in \mathbb Z$ with $\bar f_n^*(0)=0$ and $D\bar f_n^*(0)=0$ for each $n\in \mathbb Z$, $*=a, b$.

Notice that the Sacker-Sell spectra of $(\diag (A_{n-1}^1, \ldots, A_{n-1}^{k-2}))_{n\in \mathbb Z}$, $(A_{n-1}^{k-1})_{n\in \mathbb Z}$, and $(A_n^k)_{n\in \mathbb Z}$ are given by $[a_1, b_1]\cup \ldots \cup [a_{k-2}, b_{k-2}]$, $[a_{k-1}, b_{k-1}]$, and $[a_k, b_k]$, respectively. 

From the discussion in the preceding paragraph, we may assume that 
\[
\|D\mathbb F_{k-1}^{-1}(\mathbf 0)\rvert_{\mathcal W}\|\le 1/(e^{a_k}-\delta) \quad \text{and} \quad \|D\mathbb F_{k-1}(\mathbf 0)\rvert_{\mathcal U\times \mathcal V}\|\le e^{b_{k-1}}+\delta,
\]
with $\delta>0$ sufficiently small where $\mathcal U$, $\mathcal V$ and $\mathcal W$ are subspaces of $Y_\infty$ consisting of sequences in $\bigoplus_{i=1}^{k-2}\mathbb R^{d_i}$, $\mathbb R^{d_{k-1}}$ and $\mathbb R^{d_k}$, respectively.

The idea is now to modify the nonlinear parts
of $\mathbb F_{k-1}$ outside a neighborhood of $0$ by modifying the nonlinear parts of each $(\mathbb F_{k-1})_n$, $n\in\mathbb{Z}$ on $\mathbb{R}^d$, that is, by modifying the maps $(u, v, w)\mapsto (\bar f_{n-1}^a(u, v, w), \bar f_{n-1}^b(u, v, w), 0)$.
Then, for the associated modification $\tilde{\mathbb F}_{k-1}$ (which coincides with $\mathbb F_{k-1}$ on a neighborhood of $\mathbf 0$), we can apply \cite[Theorem 2]{dLW} to $\tilde{\mathbb F}_{k-1}^{-1}$.
The associated invariant manifolds will have the same structure as the invariant manifolds discussed in Section \ref{sec: foliations}, 
that is, they can be represented as
\[
\mathcal M=\{(\mathbf u, \mathbf v, \mathbf w): \ \mathbf u=h_1(\mathbf w), \ \mathbf v=h_2(\mathbf w)\},
\]
where $h_i(\mathbf w)_n$ depends only on $w_n$ for $i=1, 2$. Indeed, by~\cite[Theorem 2.1]{dLW} 
\[
\mathcal M=\left \{\mathbf x\in Y_\infty: \ \sup_{m\ge 0}(\|\mathbb F_{k-1}^{-m}(\mathbf x)\|q^{-m})<+\infty \right \},
\]
for some $q>1$. Take now $\mathbf w^i$, $i=1, 2$ in $\mathcal W$ and $k\in \mathbb Z$ such that $w_k^1=w_k^2$. We define $\mathbf x=(x_n)_{n\in \mathbb Z}\in Y_\infty$ by
\[
x_n=\begin{cases}
(h_1(\mathbf w^1)_n, h_2(\mathbf w^1)_n, w_n^1) & n\neq k;\\
(h_1(\mathbf w^2)_n, h_2(\mathbf w^2)_n, w_n^2) & n= k.\\
\end{cases}
\]
Arguing as in the proof of Lemma~\ref{lem: p x k depends only vk}, we find that $\mathbf x\in \mathcal M$, which then easily implies that $h_i(\mathbf w^1)_k=h_i(\mathbf w^2)_k$ for $i=1, 2$. Hence, $h_i(\mathbf w)_n$ depends only on $w_n$ for $i=1, 2$.
Consequently, $\Theta$ given by~\cite[Eq. (5.21)]{ZZJ} has the property that $\Theta(\mathbf x)_n$ depends only on $x_n$. Let
\[
\tilde{\mathbb F}_{k-1}:=\Theta\circ \bar{\mathbb F}_{k-1}\circ \Theta^{-1}.
\]
By the previous discussion $\bar{\mathbb F}_{k-1}(\mathbf x)_n$ depends only on $x_{n-1}$. 

Next, let
\[
\tilde{\Psi}_{k-1}:=\lim_{m\to \infty}\mathbb B^{-m}\Pi^v\tilde{\mathbb F}_{k-1}^m,
\]
where
$
(\mathbb B\mathbf v)_n=A_{n-1}^{k-1}v_{n-1}.
$
Observe that $\tilde{\Psi}_{k-1}(\mathbf x)_n$ depends only on $x_n$. The same comment applies to
\[
\Psi_{k-1}:=\tilde{\Psi}_{k-1}\circ \Theta
\]
and for
\[
\Phi_{k-1}(\mathbf u, \mathbf v, \mathbf w):=(\mathbf u, \Psi_{k-1}(\mathbf u, \mathbf v, \mathbf w), \mathbf w).
\]
We now set 
\[
\mathbb F_{k-2}:=\Phi_{k-1}\circ \mathbb F_{k-1}\circ \Phi_{k-1}^{-1}
\]
and continue with the above procedure as indicated in the proof of~\cite[Lemma 5]{ZZJ} to obtain $\Phi_{k-2}$ up to $\Phi_1$. Moreover, as indicated in the proof of~\cite[Lemma 10]{ZZJ}, the \emph{local} linearization  is given as the composition 
\[
\Phi:=\Phi_1\circ\cdots\circ\Phi_k,
\]
where each $\Phi_i$, $i=1,...,k$, have the property that  the $n$-th coordinate of the image depends only on the $n$-th coordinate of the argument. In particular, $\Phi$ also has this property.

\subsubsection{From local to global linearization}\label{sec: global linearization}

Our objective now is to extend the local linearization $\Phi$ from a neighborhood $U\subset Y_\infty$ to the whole space $Y_\infty$. We start noticing that, by \eqref{Df x close A}, for $\eta$ sufficiently small, $\mathbb{F}$ is a contraction. Moreover, by \eqref{fzero}, $\mathbb F(\mathbf 0)=\mathbf 0$. Thus, it is possible to choose a small neighborhood $U_0\subset U$ of origin such that $\mathbb{F} (U_0)\subset U_0$ and 
$U_0\subset {\rm int}\, \mathbb{F}^{-1}(U_0)\subset U$,
where ${\rm int} \, \mathbb{F}^{-1}(U_0)$ denotes the interior of the set $\mathbb{F}^{-1}(U_0)$. In particular,
\begin{equation}\label{eq: lin on U0}
    \mathbb A^\ast \circ \Phi (\mathbf x)=\Phi \circ \mathbb F(\mathbf x) \; \text{ for every } \mathbf{x}\in U_0.
\end{equation}
Then, define
\[
X_i:=\mathbb{F}^{-(i+1)}(U_0)\backslash \mathbb{F}^{-i}(U_0),\quad V_0:=\Phi(U_0),\quad Z_i:=(\mathbb{A}^\ast)^{-(i+1)} (V_0)\backslash (\mathbb{A}^\ast)^{-i}(V_0)
\]
for all $i\in \Z^+$.
It is clear that
\begin{align*}
X_i\cap X_j=\emptyset, ~~ \forall i\ne j,\quad ~~ X_i\cap U_0=\emptyset, \quad &U_0\cup \bigcup_{i=0}^{\infty}X_i=Y_\infty,~ ~\mathbb{F}^{-1}(X_i)=X_{i+1}
\end{align*}
and
\begin{align*}
Z_i\cap Z_j=\emptyset,  ~~\forall i\ne j, \quad  ~~Z_i\cap V_0=\emptyset,  \quad &V_0\cup \bigcup_{i=0}^{\infty}Z_i=Y_\infty,   ~~(\mathbb{A}^\ast)^{-1}(Z_i)=Z_{i+1}.
\end{align*}
Then we define the global conjugacy $\Phi_*\colon Y_\infty\to Y_\infty$ by
\[
\Phi_*(\mathbf x):=
\left\{\begin{array}{lll}
\Phi(\mathbf x), & \quad\forall \mathbf x\in U_0,
\\
((\mathbb A^\ast )^{-(i+1)}\circ \Phi \circ \mathbb F^{i+1}) (\mathbf x), & \quad\forall \mathbf x\in X_i, ~~\forall i\in \Z^+.
\end{array}\right.
\]

In order to show that $\Phi_*$ is a $C^1$ diffeomorphism, we note that 
\begin{equation}\label{eq: aux F-1 U0}
\Phi_*(\mathbf x)=((\mathbb A^\ast)^{-1}\circ \Phi \circ \mathbb F) (\mathbf x), \quad \forall \mathbf x\in \mathbb F^{-1}(U_0).
\end{equation}
In fact, the above equality clearly holds for $\mathbf x\in \mathbb F^{-1}(U_0)\setminus U_0=X_0$ by the definition of $\Phi_*$. Moreover, for $\mathbf x\in U_0$, using \eqref{eq: lin on U0}, we have
\[
\Phi_*(\mathbf x)=\Phi(\mathbf x)=((\mathbb A^\ast)^{-1}\circ \Phi \circ \mathbb F)(\mathbf x)
\]
which proves \eqref{eq: aux F-1 U0}.
Similarly,
\[
\Phi_*(\mathbf x)=((\mathbb A^\ast)^{-2}\circ \Phi \circ \mathbb F^2) (\mathbf x) \quad \forall \mathbf x\in \mathbb F^{-2}(U_0).
\]
Indeed, the above equality holds for $\mathbf x\in X_1$ by the definition of $\Phi_*$. Now, for $\mathbf x\in \mathbb F^{-1}(U_0)$, by \eqref{eq: aux F-1 U0}, we know that
\[
\Phi_*(\mathbf x)=((\mathbb A^\ast)^{-1}\circ \Phi \circ \mathbb F)(\mathbf x).
\]
On the other hand, for $\mathbf x\in \mathbb F^{-1}(U_0)$ we have that $\mathbb F(\mathbf x)\in U_0$. Thus, by \eqref{eq: lin on U0} applied to $\mathbb F(\mathbf{x})$,
\[\Phi\circ \mathbb F(\mathbf x)=((\mathbb A^\ast)^{-1}\circ \Phi \circ \mathbb F)(\mathbb F(\mathbf x)).\]
Therefore, combining these two observations, we conclude that
\[
\Phi_*(\mathbf x)=((\mathbb A^\ast)^{-2}\circ \Phi \circ \mathbb F^2)(\mathbf x)
\] 
as claimed.
Proceeding in the same manner, we find that 
\begin{equation}\label{1013}
\Phi_*(\mathbf x)=((\mathbb A^\ast)^{-i}\circ \Phi \circ \mathbb F^i)(\mathbf x), \quad \forall \mathbf x\in \mathbb F^{-i}(U_0), \ \forall i\in \Z^+ .
\end{equation}
This implies that $\Phi_*$ is $C^1$ on each of the open sets 
$\mathbb F^{-i}(U_0)$, $i\in \Z^+$ whose union is the entire $Y_\infty$. Consequently, $\Phi_*$ is $C^1$ on $Y_\infty$. Moreover, by \eqref{1013}, 
\[
\Phi_*\circ \mathbb{F}^{-i}(U_0)= ((\mathbb A^\ast)^{-i}\circ \Phi)(U_0)= (\mathbb A^\ast)^{-i}(V_0).
\]
Thus, $\Phi_*(X_i)=Z_i$ for all $i\in \Z^+$ which implies that $\Phi_*\colon Y_\infty\to Y_\infty$ is one-to-one and, therefore, a $C^1$ diffeomorphism. Furthermore, it follows easily from the previous observations that $\mathbb A^\ast \circ \Phi_\ast =\Phi_\ast \circ \mathbb F$. In particular, $\Phi_\ast$ is a (global) $C^1$ linearization of $\mathbb F$.

Finally, we note that $\Phi_*$ has the property that $\Phi_* (\mathbf x)_n$ depends only on $x_n$. Indeed, take $\mathbf x, \mathbf y\in Y_\infty$ with $x_n=y_n$ and let $k\in \mathbb \Z^+$ be such that $\mathbb F^k(\mathbf x)$ and $\mathbb F^k(\mathbf y)$  belong to $U_0$ (recall that $\mathbb F$ is a contraction with $\mathbb F(\mathbf 0)=\mathbf 0$). By~\eqref{1013}, we have 
\begin{equation}\label{eq: Phi ast n coord}
\Phi_*(\mathbf x)=((\mathbb A^\ast)^{-k}\circ \Phi \circ \mathbb F^k)(\mathbf x) \;\text{ and }\; 
\Phi_*(\mathbf y)=((\mathbb A^\ast)^{-k}\circ \Phi \circ \mathbb F^k)(\mathbf y).
\end{equation}
Thus, since $\mathbb F^k (\mathbf x)_{n+k}=\mathbb F^k(\mathbf y)_{n+k}$, and recalling that $\mathbb A$ (resp. $\Phi$) have the property that the $j$-th coordinate of the image depends only on the $(j-1)$-th (resp. $j$-th) coordinate of the argument, it follows by \eqref{eq: Phi ast n coord} that $\Phi_* (\mathbf x)_{n}=\Phi_* (\mathbf y)_{n}$ as claimed.

\subsection{Conclusion} Let $\mathbb A^\ast_-, \mathbb A^\ast_+, \mathbb F_-$ and $\mathbb F_+$ be as in the end of Section \ref{splitting}. Note that, by \eqref{Df x close A}, for $\eta$ sufficiently small, $\mathbb F_-$ is a contractive system, while $\mathbb F_+$ is an expanding system. Thus, from the discussion in Section \ref{sec: linear contraction} it follows that there exist $C^1$ diffeomorphisms $\Phi_{-}\colon Y_\infty^s \to Y_\infty^s$ and $\Phi_+\colon Y_\infty^u \to Y_\infty^u$ such that 
\[
\mathbb A^\ast_-\circ \Phi_-=\Phi_-\circ \mathbb F_{-} \quad \text{and} \quad \mathbb A^\ast_+\circ \Phi_+=\Phi_+\circ \mathbb F_{+}.
\]
In addition, $\Phi_{\pm}(\mathbf x)_n$ depends only on $x_n$.

Setting 
\begin{equation}\label{eq: def Phi final}
\Phi :=(\Phi_-\circ \Pi^s+\Phi_+\circ \Pi^u)\circ \eta,
\end{equation}
where $\eta$ is as in Section \ref{splitting}, we find that $\Phi \colon Y_\infty\to Y_\infty$ is a homeomorphism such that (recall \eqref{lin1})
\[
\Phi \circ \mathbb F=\mathbb A^\ast \circ \Phi.
\]
Moreover, $\Phi(\mathbf x)_n$ depends only on $x_n$ once every object involved in the definition of $\Phi$ has this property (as observed in Sections \ref{splitting} and \ref{sec: linear contraction}).  In addition, $\Phi$ is of class $C^1$ as $\Phi_{\pm}$ and $\eta$ are of class $C^1$.

For $n\in \mathbb Z$, let $h_n\colon \R^d \to \R ^d$ be given by 
\[
h_n(v):=\Phi(\mathbf v^n)_n,
\]
where $\mathbf v^n=(v_m^n)_{m\in \mathbb Z}$ is given by $v_m^n=0$ for $m\neq n$ and $v_n^n=v$. Note that, by construction, the conjugacy $\Phi$ given by \eqref{eq: def Phi final} coincides with the map $\Phi$ used in the proof of Theorem \ref{Ap} that comes from \cite[Appendix A]{DZZ}. In particular, $h_n$ given above coincide with $h_n$ given in \eqref{eq: def hn}. Thus, all we have to do to finish the proof of Proposition \ref{prop: hn is homeo} is to show that each $h_n$ is a bijection. We start showing that $h_n$ is injective. 

Let $v, w\in \R^d$ be such that $h_n(v)=h_n(w)$. Then $\Phi(\mathbf v^n)_n=\Phi(\mathbf w^n)_n$. On the other hand, since $v_m^n=w_m^n=0$ for $m\neq n$, we have $\Phi(\mathbf v^n)_m=\Phi(\mathbf w^n)_m=0$ for each $m\neq n$ (recall that $\Phi(\mathbf x)_n$ depends only on $x_n$). Hence, $\Phi(\mathbf v^n)=\Phi(\mathbf w^n)$, and thus, since $\Phi$ is injective, $\mathbf v^n=\mathbf w^n$. Consequently,  $v=v_n^n=w_n^n=w$ and $h_n$ is injective. Let us now show that $h_n$ in also onto. Take an arbitrary $w\in \R^d$ and define $\mathbf w^n$ as above. Then, since $\Phi$ is onto, there is $\mathbf y\in Y_\infty$ such that $\Phi(\mathbf y)=\mathbf w^n$. Thus, taking $v:=y_n$ we have that $\Phi(\mathbf y)_n=\Phi(\mathbf v^n)_n$, which implies that $w=h_n(v)$ and $h_n$ is onto as claimed. This finishes the proof of Proposition \ref{prop: hn is homeo}.

\begin{remark}\label{855rem}
In the context of the version of Theorem~\ref{t1} for two-sided sequences (which is needed to prove Theorem~\ref{t1}), the only difference is that $p$ and $q$ (and consequently by~\cite[Lemma 3.1]{BT} also $\eta$) are locally H\"older continuous and differentiable in origin (see~\cite[Lemma 4]{DZZ20}).
\end{remark}



\medskip{\bf Acknowledgments.}
The authors are ranked alphabetically, and their contributions should be treated equally. L. Backes was partially supported by a CNPq-Brazil PQ fellowship under Grants No. 307633/2021-7 and 304806/2024-2; D. Dragi\v cevi\'c was funded by the European Union-NextGenerationEU-StatDinMatAn-uniri-iz-25-108
; W. Zhang was partially supported by  NSF-CQ \# CSTB2023NSCQ-JQX0020, NSFC \# 12271070.

\medskip {\bf Competing interests} The authors declare no competing interests.

\medskip{\bf Author Contributions.}
All authors  made substantial contributions 
 to this research study and to writing this paper.

 \medskip{\bf Data Availability}
 Datasets were not generated or analyzed during the current study.


\end{document}